\documentclass[12pt]{article}
\usepackage{natbib}

\usepackage{hyperref}

\usepackage{a4wide}
\usepackage{amscd}
\usepackage{enumerate}

\usepackage{graphicx}

\usepackage{amsmath,multirow}
\usepackage{array}
\usepackage{mathptmx}
\usepackage[utf8]{inputenc}
\usepackage{mathtools}
\usepackage{amsfonts}
\usepackage{ed}
\usepackage[mathscr]{eucal}

\usepackage{amsthm}

\theoremstyle{plain}
\newtheorem{theorem}{Theorem}[section]
\newtheorem{proposition}{Proposition}[section]
\newtheorem{corollary}{Corollary}[section]
\newtheorem{lemma}{Lemma}[section]
\newtheorem{remark}{Remark}

\numberwithin{equation}{section}

\newcommand{\rom}[1]{\uppercase\expandafter{\romannumeral #1\relax}}
\begin{document}
{
  \title{\bf Recurrence algorithms of waiting time for the success run of length $k$ in relation to generalized Fibonacci sequences}
 \author{Jungtaek Oh\thanks{Corresponding Author:
 (e-mail: jungtaekoh0191@gmail.com)}, Chongjin Park and Sungsu Kim\hspace{.2cm}\\
}
  \maketitle
} 

\begin{abstract}
Let $V(k)$ denote the waiting time, the number of trials needed to get a consecutive $k$ ones. We propose recurrence
algorithms for the probability distribution function (pdf) and the probability generating function (pgf) of $V(k)$ in
sequences of independent and Markov dependent Bernoulli trials  using generalized Fibonacci sequences of order $k$.
Maximum likelihood estimation (MLE) methods for the probability distributions are presented in both cases with
simulation examples.
\end{abstract}

\noindent%
\tableofcontents

%



\section{Introduction}
Let $V(k)$ denote the waiting time, the number of trials needed, to get consecutive $k$ ones which is called \textit{geometric distribution of order} $k$. This definition is due to Philippou, Georghiou and Philippou
(1983).  The waiting time for the first success run of length $k$ has been studied in numerous papers dating back to De Moivre's
era (\citet{johnson2005univariate},\ pp.\ 426--432 ). For independent Bernoulli sequences, \citet{feller1968introduction} obtained a
probability model of $V(k)$. A probability model for $V(1)$ in the Markov dependent case is considered by \cite{klots1972inverse}. Related work on higher Markov chains is found in \citet{aki1996sooner}.  Relations to Fibonacci sequences have appeared in \cite{philippou1982waiting},  \cite{uppuluri1982waiting}, \cite{koutras1996waiting}, and \cite{chaves2007waiting}. Some examples of $V(k)$ can be find in \cite{greenberg1970first} and \cite{saperstein1973occurrence}. \cite{philippou1982waiting} derived a
formula for the probability distribution function of $V(k)$ in terms of the multinomial coefficients as an extension of the Fibonacci sequence of
order $k$. \cite{uppuluri1982waiting} applied the study of waiting times to find expressions of generalized Fibonacci numbers as weighted sums of binomial coefficients. \cite{koutras1996waiting} expresses the probability distribution of $k$th
non-overlapping appearance of a pair of successes separated by a given number of failures, in terms of convolutions of a shifted version of $k$th order Fibonacci numbers and of $k$th order Fibonacci
polynomials. \cite{chaves2007waiting} express $P(V(k)=v)$ using the $k$th order Fibonacci numbers for symmetric trials,
and the first two central moments for general cases.\\

There are several ways of counting a scheme. Each counting scheme depends on different conditions: whether or not the overlapping counting is permitted, and whether or not the counting starts from scratch when a certain kind or size of run has been so far enumerated. \cite{feller1968introduction} proposed a classical counting method, once $k$ consecutive successes show up, the number of occurrences of $k$ consecutive successes is counted and the counting procedure starts anew (from scratch), called \textit{non-overlapping} counting scheme which is referred to as Type $\rom{1}$ distributions of order $k$. A second scheme can be initiated by counting a success runs of length greater than or equal to $k$ preceded and followed by a failure or by the beginning or by the end of the sequence (see. e.g. \citet{mood1940distribution}) and is usually called \textit{at least} counting scheme which is referred to as Type $\rom{2}$ distributions of order $k$. \citet{ling1988binomial} suggested the \textit{overlapping} counting scheme, an uninterrupted sequence of $m\geq k$ successes preceded and followed by a failure or by the beginning or by the end of the sequence. It accounts for $m-k+1$ success runs of length of $k$ which is referred to as Type $\rom{3}$ distributions of order $k$. \citet{mood1940distribution} suggested \textit{exact} counting scheme, a success run of length exactly $k$ preceded and succeeded by failure or by nothing which is referred to as Type $\rom{4}$ distributions of order $k$.

It is well known that the negative binomial distribution arises as the distribution of the sum of $r$ independent random variables
distributed identically as geometric. Let $X_{1},X_{2},\ldots ,X_{r}$ be i.i.d. $V(k)$ random variables, and set $T_{r,k}^{(a)}=\sum_{i=1}^{r}X_{i}$. The random variable $T_{r,k}^{(a)}$ denoted by the waiting time for the $r$-th occurrence of a success run with the counting scheme utilized $a=\rom{1}$ which indicates the "non-overlapping" counting scheme, $a=\rom{2}$ which indicates the "at least" counting scheme, and $a=\rom{3}$ which indicates the "overlapping" one, denoted as $T_{r,k}^{(\rom{1})}$, $T_{r,k}^{(\rom{2})}$ and $T_{r,k}^{(\rom{3})}$, respectively. In addition, if the underline sequence is an independent and identically distributed (i.i.d.)sequence of random variables $X_{1},X_{2},\ldots$, then distributions of $T_{r,k}^{(\rom{1})}$, $T_{r,k}^{(\rom{2})}$ and $T_{r,k}^{(\rom{3})}$ will be referred to as Type $\rom{1}$, $\rom{2}$ and $\rom{3}$ negative binomial distribution of order $k$.

According to the three aforementioned counting schemes,the random variables of the number of runs of length $k$ counted in $n$ outcomes, have three different distributions which are denoted as $N_{n,k}$, $G_{n,k}$, and $M_{n,k}$. Moreover, if the underline sequence is an independent and identically distributed (i.i.d.) sequence of random variables, $X_{1},X_{2},\ldots ,X_{n}$, then distributions of $N_{n,k}$, $G_{n,k}$, and $M_{n,k}$ will be referred to as Type $\rom{1}$, $\rom{2}$ and $\rom{3}$ binomial distributions of order $k$.

To make more clear the distinction between the aforementioned counting methods we mention by way of example that for $n=12$, the binary sequence $011111000111$ contains $N_{12,2}=3$, $G_{12,2}=2$, $M_{12,2}=6$,
 $T_{2,2}^{(\rom{1})}=5$, $T_{2,2}^{(\rom{2})}=11$, and $T_{2,2}^{(\rom{3})}=4$.

Let $N_{n}^{(a)}$, $a=\rom{1},\rom{2},\rom{3}$ be a random variable denoting the number of occurrences of runs in the sequence of $n$ trials, $N_{n}^{(a)}$, $a=\rom{1},\rom{2},\rom{3}$ which is coincident with $N_{n,k}$, $G_{n,k}$  and $M_{n,k}$, respectively. The random variable $N_{n}$ is closely related to the random variable $T_{r,k}$ (see \cite{feller1968introduction}). We have the following dual relationship $N_{n}^{(a)}<r$ if and only if $T_{r,k}>n$. \citet{koutras1997waiting} and \citet{chang2012distribution} studied this dual relationship, the double generating function of $N_{n}^{(a)}$, $a=\rom{1},\rom{2},\rom{3}$ can be expressed in terms of the p.g.f. of the waiting time $T_{r,k}^{(a)}$, $a=\rom{1},\rom{2},\rom{3}$.

%
%
%
%
%
%

In this paper, we study the p.d.f. and the p.g.f. of $V(k)$ in independent and two-state Markov dependent
Bernoulli trials, and propose recursive algorithms of computing the probabilities in terms of the $k$th order Fibonacci
sequences. In Section 2, we present the results for independent Bernoulli trials, and subsequently the results for
Markov dependent Bernoulli trials.
We adopt the definition of runs of length $k$ given in \cite{feller1968introduction}. We employ the following definition for a generalized Fibonacci sequence of order $k$. Let
$F_{k}$ denote the Fibonacci sequence of order $k$ defined by $F_{k}=(f_{k1},f_{k2},\ldots)$ where $f_{ki}=0$ for $i<1$, $f_{k1}=1$, and $f_{k,i}=f_{k,i-1}+f_{k,i-2}+\cdots+f_{k,i-k}$ for $i>1$. For example, $F_1=(1,\textrm{
}1,\textrm{ }1,\textrm{ }1,\textrm{ }1,\textrm{ }1, \dots\}$, Fibonacci numbers
$F_2=(1,\textrm{ }1,\textrm{ }2,\textrm{ }3,\textrm{ }5,\textrm{ }8,\dots\}$, and
Tribonacci numbers $
F_3=(1,\textrm{ }1,\textrm{ }2,\textrm{ } 4,\textrm{ }7,$ $\textrm{
}13,\dots)$. In Section 3, we derive the distribution of $T_{r,k}^{(\rom{1})}$, $T_{r,k}^{(\rom{2})}$, and $T_{r,k}^{(\rom{3})}$ using the distribution of $V(k)$. In Section 4, we study the distributions of $N_{n,k}$, $G_{n,k}$, and $M_{n,k}$ using the dual relationship between the distribution of number of runs and the distribution of waiting time of runs.


\section{Distribution of $V(k)$}

 \subsection{$V(k)$ in independent Bernoulli trials}
$\hspace*{0.4cm}$ Consider an infinite sequence of independent Bernoulli trials $(X_i|i=1,2,\ldots)$ with probabilities given
by P($X_i=1$)$=p$ $ (0 < p < 1)$ and
P($X_i=0$)$=1-p$ $=q$. In this section, we obtain recurrences for the probability distribution function and the probability generating function of
 $V(k)$ using the generalized Fibonacci sequences, and provide the Maximum likelihood estimation (MLE) method for the parameters.

\subsubsection{Pdf of $V(k)$}
\begin{theorem}
The p.d.f. of $V(2)$ is given by
\begin{gather*}
P(V(2)=v)=h_{v-1}p^2,
\end{gather*}
 where $h_v=qh_{v-1}+pqh_{v-2}=q(ph_{v-1}+ph_{v-2})$ for $v=3,4,\dots,$
with $h_1=1$ and $h_2=q$.\end{theorem}
\begin{proof}
A derivation of the p.d.f. of $V(2)$ is given as follows:
\begin{equation*}
\begin{split}
P(V(2)=2)&=pp=p^2=h_1p^2,\qquad
P(V(2)=3)=q(pp)=h_2p^2\\
P(V(2)=4)&=(qq+pq)pp=q(p^0q+p^1\cdot 1)p^2\\&=q(p^0h_2+p^1h_1)p^2=h_3p^2\\
P(V(2)=5)&=(qqq+pqq+qpq)pp=q((q^2+pq)+pq)p^2\\&=q(p^0h_3+p^1h_2)p^2=h_4p^2\\
\vdots.
\end{split}
\end{equation*}
 Hence, the p.d.f. of $V(2)$ can be calculated recursively.
\end{proof}

\vskip 5pt
\begin{theorem}
For $k=3$, the p.d.f. of $V(k)$ is given by
\begin{gather*}
P(V(3)=v)= h_{v-2}p^3,
\end{gather*}
 where $h_v=qh_{v-1}+pqh_{v-2}+p^2qh_{v-3}=q(p^0h_{v-1}+p^1h_{v-2}+p^2h_{v-3})$ for $v=4,5,\dots,$
with $h_1=1$, $h_2=q$ and $h_3=q^2+pq$.
\end{theorem}
\begin{proof}
The p.d.f. of $V(3)$ is given by
\begin{gather*}
P(V(3)=3)=ppp=p^3=h_1p^3,\qquad
P(V(3)=4)=q(ppp)=h_2p^3\\
P(V(3)=5)=(qq+pq)ppp=q(p^0q+p^1\cdot 1)p^3=q(p^0h_2+p^1h_1)p^3=h_3p^3\\
P(V(3)=6)=(qqq+pqq+qpq+ppq)ppp=q((q^2+pq)+pq+p^2)p^3\\=q(p^0h_3+p^1h_2+p^2h_1)p^3=h_4p^3\\
\vdots .
\end{gather*}
 Hence, the p.d.f. of $V(3)$ can be calculated recursively.
\end{proof}
\vskip 5pt
Following the similar steps as in the case $k=3$, the following theorem can be easily established.
\begin{theorem}
The p.d.f. of $V(k)$ is given by
\begin{gather*}
P(V(k)=v)= h_{v-k+1}p^k,
\end{gather*}
 where $h_v=q(h_{v-1}+ph_{v-2}+p^2h_{v-3}+\cdots +p^{k-1}h_{v-k})$ for $v=k+1,k+2,\dots,$
with  $h_1=1$, $h_2=q$, $h_3=q^2+pq,$ \dots, $h_k=qh_{k-1}+pqh_{k-2}+p^2qh_{k-3}+\cdots$ $+p^{k-1}qh_0$ and
 $h_{k+1}=qh_k+pqh_{k-1}+p^2qh_{k-2}+\cdots +p^{k-1}qh_1$, \dots etc.
\end{theorem}
\begin{proof}
The p.d.f. of $V(k)$ is given by
\begin{gather*}
P(V(k)=k)=\underbrace{pp\ldots p}_{k}=p^k=h_1p^k,\qquad
P(V(k)=k+1)=q\underbrace{pp\ldots p}_{k}=h_2p^k\\
P(V(k)=k+2)=(qq+pq)\underbrace{pp\ldots p}_{k}=q(p^0q+p^1\cdot 1)p^k=q(p^0h_2+p^1h_1)p^3=h_3p^k\\
P(V(k)=k+3)=(qqq+pqq+qpq+ppq)\underbrace{pp\ldots p}_{k}=q((q^2+pq)+pq+p^2)p^k\\=q(p^0h_3+p^1h_2+p^2h_1)p^k=h_4p^k\\
\vdots .
\end{gather*}
 Hence, the p.d.f. of $V(k)$ can be calculated recursively.
\end{proof}
\vskip 5pt

\begin{theorem}
For $V(k)$, $h_v$ forms the generalized Fibonacci sequence of order $k$.
\end{theorem}
\begin{proof}
Observe that
\begin{gather*}
P(V(k)=v)=h_{v-k+1}p^k=q\sum_{i=1}^{k}p^{i-1}h_{v-(k-1)-i}\cdot p^k.
\end{gather*}
Therefore,  $h_v=q\sum_{i=1}^{k}p^{i-1}h_{v-i}$ forms the generalized Fibonacci sequence of order $k$, where the
sequence is given by
\begin{gather*}
\{qp^{v-1-0}h_0,qp^{v-1-1}h_1,qp^{v-1-2}h_2,\ldots,qp^{v-1-(v-2)}h_{v-2},qp^{v-1-(v-1)}h_{v-1},\\qp^{v-1-(v)}h_{v},\cdots\}.
\end{gather*}
Hence, the proof is complete.
\end{proof}

\subsubsection{Some numerical examples}
The exact distribution of $V(k)$ seems to be intractable, except for $p=1/2$. In this case, every equal length sequence has same probability and number of occurrences, and the distribution of $V(k)$ is nicely deriving. In this section, the pdf of $V(k)$ for a special case of $p$=1/2 is calculated numerically, using generalized Fibonacci sequences.\\
For $k=2$, using the Fibonacci sequences of order $2$, i.e. $\{1,1,2,3,5,8,13,21,34,\dots\}$. It can be shown that
\begin{gather*}
P(V(2)=v)=f_{v-1}(1/2)^{v},
\end{gather*} where $f_v=\{1,1,2,3,5,8,13,21,34,\dots\}$, $v$=2,3,\dots.\\
For $k=3$, using the Fibonacci sequences of order $3$, it can be shown that
\begin{gather*}
P(V(3)=v)=f_{v-2}(1/2)^{v}\textrm{ for }v=3,4,5,\dots,
\end{gather*}
where $f_i=\{1,1,2,4,7,13,24,44,81,\dots\}$.\\ 
For $k=4$, using the Fibonacci  sequences of order $4$, it can be shown that
\begin{gather*}
P(V(4)=v)=f_{v-3}(1/2)^{v}\textrm{ for }v=4,5,6,\dots,
\end{gather*} where $f_i=\{1,1,2,4,8,15,29,56,108,\dots\}$.\\
For $k=5$, using the Fibonacci sequences of order $5$, it can be shown that
\begin{gather*}
P(W(5)=v)= f_{v-4}(1/2)^{v}\textrm{ for }v=5,6,7,\dots,
\end{gather*}
where $f_i=\{1,1,2,4,8,16,31,61,120,\dots\}$.
In general, it can be shown that:
\begin{theorem}
For Fibonacci sequences of order $k$, $P(V(k)=v)$ is given by
\begin{gather*}
P(V(k)=v)= f_{v-k+1}(1/2)^{v}\textrm{ for }v=k,k+1,k+2,\dots,
 \end{gather*}
 where $f_{v-k+1}$ is the ${(v-k+1)}^{th}$ generalized fibonacci sequences of order $k$.
 \end{theorem}

In order to compute the probability of $V(k)$ for $p=1/2$, we consider the $f_{v-k+1}$ which is ${(v-k+1)}^{th}$ generalized fibonacci sequences of order $k$ (also known as the $k$-fold fibonacci, $k$-th order fibonacci, $k$-fibonacci or polynacci numbers).
It is worth mentioning that the Binet-style formula can be used to generate the generalized fibonacci number of order $k$ (that is, the Tribonaccis, Tetranaccis, etc.). It should be stressed that the Binet-style formula presented many other ways, as seen in
the articles \citet{ferguson1966expression, flores1967direct,gabai1970generalized,kalman1982generalized,lee2001binet,levesque1985m,miles1960generalized}.

Next, we remind the reader of the famous Binet formula (also known
as the de~Moivre formula) that
can be used to calculate $F_n$, the Fibonacci numbers:
\begin{eqnarray*}
F_n &=& \frac{1}{\sqrt{5}} \left[ \left( \frac{1+\sqrt{5}}{2}\right)^n
     - \left( \frac{1-\sqrt{5}}{2}\right)^n \right] \\[1.5ex]
     &=& \frac{\alpha^n - \beta^n}{\alpha - \beta}
\end{eqnarray*}
for $\alpha > \beta$ the two roots of $x^2 - x - 1=0$. It is convenient (and not particularly difficult) to rewrite this
formula as follows:
\begin{equation}
 F_n = \frac{\alpha-1}{2 + 3(\alpha -2)} \alpha^{n-1} +
     \frac{\beta-1}{2 + 3(\beta -2)} \beta^{n-1}
\label{Binet2}
\end{equation}

As mentioned above, \citet{spickerman1984binet} proved the following
formula for the generalized fibonacci number of order $k$:
\begin{equation}\label{sjf2}
F_n^{(k)} = \sum_{i=1}^k
       \frac{\alpha_i^{k+1} - \alpha_i^{k}}
            {2\alpha_i^k - (k+1)} \alpha_i^{n-1}
\end{equation}
Note that the set $\{\alpha_i\}$ is the set of roots of
$x^k - x^{k-1} - \cdots - 1=0$.

After then, \citet{dresden2014simplified} developed more simple version of \eqref{sjf2} to present the generalized fibonacci number of order $k$. Theirs result is the following representation of $F_n^{(k)}$:

\begin{theorem}
\label{T1}
 For $F_n^{(k)}$ the
$n^{\mbox{th}}$ $k$-generalized Fibonacci number, then
\begin{equation}\label{fT1}
F_n^{(k)} = \sum_{i=1}^k \frac{\alpha_i - 1}{2 + (k+1)(\alpha_i - 2)} \alpha_i^{n-1}
\end{equation}
for $\alpha_1, \dots, \alpha_k$ the roots of $x^k - x^{k-1} - \cdots - 1=0$.
\end{theorem}

Note that for $k=2$, Eq.~(\ref{fT1}) reduces to
the variant of the Binet formula (for the standard Fibonacci numbers) from
Eq.~(\ref{Binet2}).

\subsubsection{Statistical Inference for the pdf of $V(k)$ in the independent case}
$\hspace*{0.4cm}$ In this section, we investigate the maximum likelihood estimation (MLE) method for the p.d.f. of $V(k)$. We present the result for the p.d.f. of $V(2)$, and
the higher order cases can be established similarly.
\vskip 5pt
\begin{theorem}
The p.d.f. of $V(2)$ is given by
\begin{gather*}
P(V(2)=2)=p\left[\frac{R_1^{v-1}}{(R_1-R_2)}+\frac{R_2^{v-1}}{(q-2R_1)}\right],\qquad
v=2,3,\dots,\end{gather*}
where $R_1=\frac{q+\sqrt{q^2+4pq}}{2}$ and $R_2=\frac{q-\sqrt{q^2+4pq}}{2}$.
\end{theorem}
\begin{proof}
Using the second order linear recurrence relation given by $h_v=qh_{v-1}$ $+pqh_{v-2}$, for $v=2,3,\dots$, with the
initial conditions, $h_1=1$ and $h_2=q$. To solve the second order linear recurrence relation, transfer all the terms to the left-hand side

\begin{equation}
\label{2ndorreqiid}
h_v-qh_{v-1}-pqh_{v-2}=0
\end{equation}
The zero on the right-hand side signifies that this is a homogeneous difference equation. As a solution, try $h_{v}=Ax^{v}$, where $x$ and $A$ are constants. Substituting $h_{v}=Ax^{v}$, $h_{v-1}=Ax^{v-1}$ and $h_{v-2}=Ax^{v-2}$ into Eq.\eqref{2ndorreqiid} gives

\begin{equation*}
\begin{split}
Ax^{v}-qAx^{v-1}-pqAx^{v-2}=0\\
Ax^{v-2}(x^{2}-qx-pq)=0
\end{split}
\end{equation*}
If $x=0$, or $A=0$, then Eq.\eqref{2ndorreqiid} has trivial solutions (i.e. $h_{v}=0$). Otherwise, if $x\neq0$ and $A\neq0$ then
\begin{equation*}
\begin{split}
x^{2}-qx-pq=0.
\end{split}
\end{equation*}
This is the auxiliary equation of Eq.\eqref{2ndorreqiid}. Being a quadratic, the auxiliary equation signifies that the difference equation is of second order. Now, the auxiliary equation has the two roots, which means that
\begin{equation*}
\begin{split}
x_{1}=\frac{q+\sqrt{q^{2}+4pq}}{2}\ \text{and}\ x_{2}=\frac{q-\sqrt{q^{2}+4pq}}{2}
\end{split}
\end{equation*}

The general solution to the equation is therefore

\begin{equation*}
\begin{split}
h_{v}=A_{1}\left(\frac{q+\sqrt{q^{2}+4pq}}{2}\right)^{v}+A_{2}\left(\frac{q-\sqrt{q^{2}+4pq}}{2}\right)^{v}.
\end{split}
\end{equation*}

When $n=1$, $h_{1}=1$ and since $h_{1}=A_{1}\left(\frac{q+\sqrt{q^{2}+4pq}}{2}\right)^{1}+A_{2}\left(\frac{q-\sqrt{q^{2}+4pq}}{2}\right)^{1}$, then
\begin{equation}
\label{2ndorreqiid-sub1}
\begin{split}
A_{1}\left(\frac{q+\sqrt{q^{2}+4pq}}{2}\right)+A_{2}\left(\frac{q-\sqrt{q^{2}+4pq}}{2}\right)=&1\\
(A_{1}+A_{2})\frac{q}{2}+(A_{1}-A_{2})\frac{\sqrt{q^{2}+4pq}}{2}=&1
\end{split}
\end{equation}

When $n=2$, $h_{2}=q$ and since $h_{2}=A_{1}\left(\frac{q+\sqrt{q^{2}+4pq}}{2}\right)^{2}+A_{2}\left(\frac{q-\sqrt{q^{2}+4pq}}{2}\right)^{2}$, then
\begin{equation}
\label{2ndorreqiid-sub2}
\begin{split}
A_{1}\left(\frac{q+\sqrt{q^{2}+4pq}}{2}\right)^{2}+A_{2}\left(\frac{q-\sqrt{q^{2}+4pq}}{2}\right)^{2}=&q\\
(A_{1}+A_{2})\frac{q^{2}+2pq}{2}+(A_{1}-A_{2})\frac{q\sqrt{q^{2}+4pq}}{2}=&q
\end{split}
\end{equation}

Multiplying $q$ Eq. \eqref{2ndorreqiid-sub1}, and subtract from Eq. \eqref{2ndorreqiid-sub1} gives
\begin{equation}
\label{2ndorreqiid-sub3}
\begin{split}
(A_{1}+A_{2})\frac{2pq}{2}=0\\
\Longrightarrow A_{1}=-A_{2}
\end{split}
\end{equation}

\begin{equation}
\label{2ndorreqiid-sub4}
\begin{split}
-A_{2}\left(\frac{q+\sqrt{q^{2}+4pq}}{2}\right)+A_{2}\left(\frac{q-\sqrt{q^{2}+4pq}}{2}\right)=&1\\
-2A_{2}\left(\frac{\sqrt{q^{2}+4pq}}{2}\right)=&1\\
\end{split}
\end{equation}

\begin{equation}
\label{}
\begin{split}
A_{2}=-\frac{1}{\sqrt{q^{2}+4pq}}\\
A_{1}=\frac{1}{\sqrt{q^{2}+4pq}}\\
\end{split}
\end{equation}

Solving the auxiliary quadratic equation $x^2=qx+pq$, the general solution is given by
\begin{gather*}
h_v=\frac{(q+\sqrt{q^2+4pq})^v}{2^v\sqrt{q^2+4pq}}+\frac{(q-\sqrt{q^2+4pq})^{v+1}}{2^v(q^2+4pq-q\sqrt{q^2+4pq})},
\qquad v=2,3,\dots.
\end{gather*}
Now, the theorem follows. Equivalently, note that we may substitute $1-p$ for $q$, resulting in the expression of the
p.d.f. in one parameter only.
\end{proof}
In order to obtain ML estimates, we first generate 200 independent Bernoulli samples using $p=0.1,0.3,0.5,0.7,0.9$.
Since the likelihood function and the first order equations are too complex to analytically work with, we use R with
'optim' function, where the Nelder-Mead method was employed. The results are shown in Table 1. It is shown that the
MLE method perform reasonably well for the range of true values of $p$. Under some regularity conditions, $\hat{p}$
possesses the properties of the MLE, i.e. they are consistent and has an asymptotic normal distribution. Standard
deviations are shown in the parenthesis, which were calculated using the Bootstrap method.
\begin{table}[h]
\caption{ML estimates for $p$ in the independent cases}
\begin{center}\begin{tabular}{cccccc}\hline
True Value & 0.1&0.3&0.5&0.7&0.9\\
\hline
ML Estimates & 0.099&0.270&0.533&0.769&0.872\\
St. Dev. & (0.01)&(0.03)&(0.03)&(0.02)&(0.02)\\
\hline
\end{tabular}\end{center}
\end{table}
\subsubsection{Pgf of $V(k)$}
$\hspace*{0.4cm}$In this section, we present the closed form solutions to the p.g.f.s.
\begin{theorem}\label{pgf:v(2),iid}
The p.g.f. of $V(2)$ is given by
\begin{gather*}
G_{V(2)}(s)=\frac{s^2p^2}{1-qs-qps^2}.
\end{gather*}
\end{theorem}
\begin{proof}
P.g.f. of $V(2)$ is given by
\begin{gather}
G_{V(2)}(s)=s^2p^2 +qp^2s^3+h_3p^2s^4+h_4p^2s^5+\cdots.
\end{gather}
First, multiply $qs$ on both sides of (1) to obtain
\begin{gather}
qsG_{V(2)}(s)=s^3qp^2 +qqp^2s^4+qh_3p^2s^5+qh_4p^2s^6+\cdots.
\end{gather}
Next, multiply $qps^2$ on both sides of (1) to obtain
\begin{gather}
qps^2G_{V(2)}(s)=s^4qpp^2 +qpqp^2s^5+qph_3p^2s^6+qph_4p^2s^7+\cdots.
\end{gather}
Now, calculate subtract the second and third equation from the first to obtain
\begin{gather*} (1-qs-qps^2)G_{V(2)}(s)
=s^2p^2 +* s^3+* s^4+* s^5+\cdots,\end{gather*} and note that each
term in $*$ is zero because $(h_v-qh_{v-1}-qph_{v-2})=0$ by the defining
relations of Fibonacci sequences.
\end{proof}

From Theorem we get the following Corollary.

\begin{corollary}The probability mass function of $V(2)$ satisfies the recursive scheme
\begin{equation*}\label{eq: 1.1}
\begin{split}
P(V(2)=v)=\beta P(V(2)=v-1)+(1-\alpha)(1-\beta)P(V(2)=v-2),\ v\geq 3,
\end{split}
\end{equation*}
with initial conditions $P(V(2)=0)=P(V(2)=1)=0$ and $P(V(2)=2)=p^{2}$.
\end{corollary}

\begin{proof}
It follows by equating the coefficients of $s^{v}$ on both sides of
\begin{equation*}
(1-q s-qps^{2})\sum_{v=2}^{\infty}P(V(2)=v)s^{v}=ps^2
\end{equation*}
we may easily written as
\begin{equation*}\label{eq: 1.1}
\begin{split}
P(V(2)=v)s^{v}-q sP(V(2)=v-1)s^{v-1}-qps^{2}P(V(2)=v-2)s^{v-2}=0\\
\left\{P(V(2)=v)-q P(V(2)=v-1)-qpP(V(2)=v-2)\right\}s^{v}=0
\end{split}
\end{equation*}
We may deduce the following recursive scheme.
\begin{equation*}\label{eq: 1.1}
\begin{split}
P(V(2)=v)=q P(V(2)=v-1)+qpP(V(2)=v-2).
\end{split}
\end{equation*}
\end{proof}

Following the similar steps as in the case $k=2$, the following theorems can be easily verified.
\begin{theorem}
The p.g.f. of $V(3)$ is given by
\begin{gather*}
G_{V(3)}(s)=\frac{s^3p^3}{1-qs-qps^2-qp^2s^3}.
\end{gather*}
\end{theorem}

From Theorem we get the following Corollary.

\begin{corollary}The probability mass function of $V(3)$ satisfies the recursive scheme
\begin{equation*}\label{eq: 1.1}
\begin{split}
P(V(3)=v)=q P(V(3)=v-1)+qpP(V(3)=v-2)+qp^{2}P(V(3)=v-3)\ \text{for}\ v\geq 4,
\end{split}
\end{equation*}
with initial conditions $P(V(3)=0)=P(V(3)=1)=P(V(3)=2)=0$ and $P(V(3)=3)=p^{3}$.
\end{corollary}

\begin{theorem}
The p.g.f. of $V(k)$ is given by
\begin{equation*}\label{pgf:v(k):iid}
\begin{split}
G_{V(k)}(s)=\frac{s^kp^k}{1-q\sum_{i=1}^kp^{i-1}s^i}=\frac{s^kp^k-p^{k+1}s^{k+1}}{1-s+p^{k}qs^{k+1}}\\
\end{split}
\end{equation*}
\end{theorem}

From Theorem \ref{pgf:v(k):iid} we get the following Corollary.

\begin{corollary}The probability mass function of $V(k)$ satisfies the recursive scheme
\begin{equation*}\label{eq: 1.1}
\begin{split}
P(V(k)=v)=P(V(k)=v-1)-p^{k}qP(V(k)=v-k-1),\ v> k+1,
\end{split}
\end{equation*}
with initial conditions
\begin{equation*}
\begin{split}
P(V(k)=v)=\left\{
  \begin{array}{ll}
    0,  &\ \text{if}\ \ 0\leq v<k, \\
    p^{k}, &\ \text{if}\ \ v=k, \\
    qp^{k}, &\ \text{if}\ \ v=k+1.\\
  \end{array}
\right.
\end{split}
\end{equation*}
\end{corollary}

From Theorem \ref{pgf:v(k):iid} we get the different type recursive scheme.
\begin{remark}
The probability mass function of $V(k)$ satisfies the recursive scheme
\begin{equation*}\label{eq: 1.1}
\begin{split}
P(V(k)=v)=q\sum_{i=1}^{k}p^{i-1}P(V(k)=v-i), \ v> 2k,
\end{split}
\end{equation*}
with initial conditions
\begin{equation*}
\begin{split}
P(V(k)=v)=\left\{
  \begin{array}{ll}
    0  & \text{if $0\leq v<k$,} \\
    p^{k} & \text{if $v=k$,} \\
    qp^{k} & \text{if $k+1\leq v\leq 2k$.}\\
  \end{array}
\right.
\end{split}
\end{equation*}
\end{remark}

\subsection{$V(k)$ in two state Markov dependent trials}
$\hspace*{0.4cm}$ Consider an infinite sequence of two state Markov dependent Bernoulli trials $\{X_i|i=1,2,\dots\}$ with
P($X_i=1$)$=p$, P($X_i=0$)$=1-p=q$, P($X_i=1|X_{i-1}=1$)$=\alpha$ and P($X_i=0|X_{i-1}=0$)$=\beta$ for $i > 1$. In this section,
the p.d.f. and the p.g.f. of
 $V(k)$ are obtained using the generalized Fibonacci sequences, and estimate the parameters by the Maximum Likelihood Estimation (MLE) method.
\subsubsection{Pdf of $V(k)$}
\begin{theorem}
For $k=2$, the p.d.f. of $V(k)$ is given by
\begin{gather*}
P(V(2)=v)=h_{v-1}\alpha,
\end{gather*}
 where $h_v=\beta h_{v-1}+(1-\alpha)(1-\beta)h_{v-2}$ for v=3,4,\dots,
with $h_1=p$ and $h_2=q(1-\beta)$.\end{theorem}
\begin{proof}
A derivation of the p.d.f. of $V(2)$ is given as follows.
\begin{gather*}
P(V(2)=2)=p\alpha=h_1\alpha,\qquad
P(V(2)=3)=q(1-\beta)\alpha=h_2\alpha\\
P(V(2)=4)=q\beta(1-\beta)\alpha+p(1-\alpha)(1-\beta)\alpha=(\beta h_2+(1-\alpha)(1-\beta)h_1)\alpha\\=h_3\alpha\\
\vdots .
\end{gather*}
 Hence, the p.d.f. of $V(2)$ can be calculated recursively.
\end{proof}
\vskip 5pt
\begin{theorem}
For $k=3$, the p.d.f. of $V(k)$ is given by
\begin{gather*}
P(V(3)=v)= h_{v-2}{\alpha}^2,
\end{gather*}
 where $h_v=\beta h_{v-1}+(1-\alpha)(1-\beta)h_{v-2}+\alpha(1-\alpha)(1-\beta)h_{v-3}$ for v=4,5,\dots,
with $h_1=p$, $h_2=q(1-\beta)$ and $h_3=\beta h_2+(1-\alpha)(1-\beta)h_1$.
\end{theorem}
\begin{proof}
A derivation of the p.d.f. of $V(3)$ is given as follows.
\begin{gather*}
P(V(3)=3)=p{\alpha}^2=h_1{\alpha}^2,\qquad
P(V(3)=4)=q(1-\beta){\alpha}^2=h_2{\alpha}^2\\
P(V(3)=5)=q\beta(1-\beta)\alpha\alpha+p(1-\alpha)(1-\beta)\alpha\alpha\\=(\beta
h_2+(1-\alpha)(1-\beta)h_1)\alpha=h_3\alpha^{2}\\
P(V(3)=6)=q\beta\beta(1-\beta)\alpha\alpha+p(1-\alpha)\beta(1-\beta)\alpha\alpha
+\\q(1-\beta)(1-\alpha)(1-\beta)\alpha\alpha+p\alpha(1-\alpha)(1-\beta)\alpha\alpha=\\(\beta
h_3+(1-\alpha)(1-\beta)h_2+\alpha(1-\alpha)(1-\beta)h_1)\alpha^{2}=h_4\alpha^{2}\\
\vdots .
\end{gather*}
 Hence, the p.d.f. of $V(3)$ can be calculated recursively.
\end{proof}
Following the similar steps as in the case $k=3$, the following theorem can be easily established.
\vskip 5pt
\begin{theorem}
The p.d.f. of $V(k)$ is given by
\begin{gather*}
P(V(k)=v)= h_{v-k+1}{\alpha}^{k-1},
\end{gather*}
 where $h_v=\beta h_{v-1}+(1-\alpha)(1-\beta)\sum_{i=0}^{k-2}{\alpha}^ih_{v-i-2}$ for v=k+1,k+2,\dots,
with  $h_1=p$, $h_2=q(1-\beta)$, $h_3=\beta h_2+(1-\alpha)(1-\beta)h_1$, $h_4=\beta
h_3+(1-\alpha)(1-\beta)h_2$ $+\alpha(1-\alpha)(1-\beta)h_1$, \dots, and $h_k=\beta
h_{k-1}+(1-\alpha)(1-\beta)\sum_{i=0}^{k-3}{\alpha}^i h_{k-i-2}$.
\end{theorem}
\begin{proof}
A derivation of the p.d.f. of $V(k)$ is given as follows.
\begin{gather*}
P(V(k)=k)=p{\alpha}^{k-1}=h_1{\alpha}^{k-1},\qquad
P(V(k)=k+1)=q(1-\beta){\alpha}^{k-1}=h_2{\alpha}^{k-1}\\
P(V(k)=k+2)=q\beta(1-\beta){\alpha}^{k-1}+p(1-\alpha)(1-\beta){\alpha}^{k-1}\\=(\beta
h_2+(1-\alpha)(1-\beta)h_1)\alpha=h_3{\alpha}^{k-1}\\
P(V(3)=6)=q\beta\beta(1-\beta){\alpha}^{k-1}+p(1-\alpha)\beta(1-\beta){\alpha}^{k-1}
+\\q(1-\beta)(1-\alpha)(1-\beta){\alpha}^{k-1}+p\alpha(1-\alpha)(1-\beta){\alpha}^{k-1}=\\(\beta
h_3+(1-\alpha)(1-\beta)h_2+\alpha(1-\alpha)(1-\beta)h_1){\alpha}^{k-1}=h_4{\alpha}^{k-1}\\
\vdots .
\end{gather*}
 Hence, the p.d.f. of $V(k)$ can be calculated recursively.
\end{proof}
\vskip 5pt
It is noted that the independence case is a special case of the Markov dependent case.
\subsubsection{Statistical Inference for the pdf of $V(k)$ in two state Markov dependent cases}
$\hspace*{0.4cm}$ In this section, we investigate the MLE method for the p.d.f. of $V(k)$. As in the independent cases, we present the
result for the p.d.f. of $V(2)$, and the higher order cases can be established similarly.
\vskip 5pt
\begin{theorem}
The p.d.f. of $V(2)$ is given by
\begin{gather*}
P(V(2)=2)=\alpha\left[\frac{R_2(pR_1+q-\beta)}{R_2(R_1-R_2)}R_1^{v-1}
+\frac{q(1-\beta)-pR_1}{q-2R_1}R_2^{v-1}\right],\\
v=2,3,\dots,\end{gather*}\\
where $R_1=\frac{\beta+\sqrt{\beta^2+4(1-\alpha)(1-\beta)}}{2}$ and
$R_2=\frac{\beta-\sqrt{\beta^2+4(1-\alpha)(1-\beta)}}{2}$.
\end{theorem}
\begin{proof}
Using the second order linear recurrence relation given by
\begin{equation*}\label{eq: 1.1}
h_v=\beta h_{v-1}+(1-\alpha)(1-\beta)h_{v-2}
\end{equation*}
 for $v=2,3,\dots$, with the initial conditions, $h_1=p$ and $h_2=q(1-\beta)$, and solving the auxiliary quadratic equation
$x^2=\beta x+(1-\alpha)(1-\beta)$, we can derives by second order difference equation
\begin{gather*}
h_v=\frac{R_2(pR_1+q-\beta)}{R_2(R_1-R_2)}R_1^{v-1}
+\frac{q(1-\beta)-pR_1}{(q-2R_1)}R_2^{v-1}\qquad v=2,3,\dots.
\end{gather*}
Now the theorem follows. Equivalently, note that we may substitute $1-p$ for $q$ and
$p=\frac{1-\beta}{2-\alpha-\beta}$, resulting in the expression of the p.d.f. in two parameters only.
\end{proof}
In order to obtain ML estimates, we first generate 200 two state Markov dependent Bernoulli samples using
$\alpha=0.1,0.3,0.5,0.7,0.9$ and $\beta=0.1,0.3,0.5,0.7,0.9$. Corresponding $p$ values for each combination of
$\alpha$ and $\beta$ are shown in Table 2. \\
\begin{table}[h]
\caption{$p$ values for $(\alpha,\beta)$ using $p=\frac{1-\beta}{2-\alpha-\beta}$}
\begin{center}\begin{tabular}{ccccccc}\hline
&&\multicolumn{5}{c}{$\beta$}\\
&& 0.1&0.3&0.5&0.7&0.9\\\hline
\multirow{5}{*}{$\alpha$}&0.1&0.500&0.438&0.357&0.250&0.100\\
&0.3&0.563&0.500&0.417&0.300&0.125\\
&0.5&0.643&0.583&0.500&0.375&0.167\\
&0.7&0.750&0.700&0.625&0.500&0.250\\
&0.9&0.900&0.875&0.833&0.750&0.500
\\\hline
\end{tabular}\end{center}
\end{table}
$\hspace*{0.4cm}$As in the case for independent trials, the likelihood function and the first order equations are too complex to
analytically work with. Here, we also use R with 'optim' function, where the Nelder-Mead method was employed. The
results are shown in Table 3. Corresponding $p$ values for each combination of $(\alpha,\beta)$ can be looked up in
Table 2. It shows that the MLE method perform reasonably well for the range of true values of $(p,\alpha,\beta)$.
Under some regularity conditions, these estimators are consistent and have asymptotic normal distributions. We
obtained the largest standard deviation to be 0.037 using the Bootstrap method.

\begin{table}[h]
\caption{ML estimates for $(p,\alpha,\beta)$ in two state Markov dependent cases}
\centering
\resizebox{\textwidth}{!}{
\begin{tabular}{ccccccc}\hline
&&\multicolumn{5}{c}{$\beta$}\\
&& 0.1&0.3&0.5&0.7&0.9\\\hline
\multirow{5}{*}{$\alpha$}&0.1&(0.51,0.11,0.99)&(0.46,0.09,0.28)&(0.37,0.12,0.52)&(0.24,0.11,0.72)&(0.10,0.10,0.90)\\
&0.3&(0.57,0.30,0.07)&(0.50,0.29,0.30)&(0.41,0.32,0.53)&(0.31,0.29,0.69)&(0.13,0.28,0.89)\\
&0.5&(0.65,0.54,0.12)&(0.59,0.49,0.26)&(0.52,0.52,0.47)&(0.38,0.49,0.69)&(0.16,0.55,0.91)\\
&0.7&(0.72,0.68,0.15)&(0.70,0.69,0.29)&(0.64,0.72,0.51)&(0.46,0.70,0.74)&(0.25,0.71,0.90)\\
&0.9&(0.89,0.89,0.11)&(0.86,0.89,0.30)&(0.85,0.92,0.53)&(0.74,0.91,0.75)&(0.50,0.91,0.91)\\
\hline
\end{tabular}}
\end{table}

\subsubsection{Pgf of $V(k)$}
In this section, we present the closed form solutions to the pgfs.
\vskip 5pt
\begin{theorem}
The p.g.f. of V(2) is given by
\begin{gather*}
\begin{split}
G_{V(2)}(s)=\frac{p\alpha s^2+(q-\beta)\alpha s^3}{1-\beta s-(1-\alpha)(1-\beta)s^2}.
\end{split}
\end{gather*}
\end{theorem}
\begin{proof}
P.g.f. of V(2) is given by
\begin{equation*}\label{eq: 1.1}
\begin{split}
G_{V(2)}(s)=&p\alpha s^2+q(1-\beta)\alpha s^3+h_3\alpha s^4+h_4\alpha s^5+\cdots\\
=&\alpha(p s^2+q(1-\beta)s^3+h_3 s^4+h_4 s^5+\cdots).
\end{split}
\end{equation*}
Letting $H_2(s)=p s^2+q(1-\beta)s^3+h_3 s^4+h_4 s^5+\cdots$, first, multiply $-\beta s$ on both sides of $H_2(s)$ to obtain
\begin{gather*}
-\beta sH_2(s)=-\beta p s^3-\beta q(1-\beta)s^4-\beta h_3 s^5 -\beta h_4 s^6-\cdots.
\end{gather*}
Next, multiply $-(1-\alpha)(1-\beta)s^2$ on both sides of $H_2(s)$ to obtain
\begin{equation*}\label{eq: 1.1}
\begin{split}
-(1-\alpha)(1-\beta)s^2H_2(s)=&-(1-\alpha)(1-\beta) p s^4-(1-\alpha)(1-\beta)q(1-\beta)s^5\\
&-(1-\alpha)(1-\beta)h_3 s^6-(1-\alpha)(1-\beta)h_4 s^7-\cdots.
\end{split}
\end{equation*}
Now, calculate $(1-\beta s-(1-\alpha)(1-\beta)s^2)H_2(s)$:
\begin{gather*}
(1-\beta s-(1-\alpha)(1-\beta)s^2)H_2(s)
=h_1 s^2+(h_{2}-\beta h_{1})s^3 +*s^4\\+*s^5+*s^6+\cdots,\end{gather*} and note that each
term in $*$ is zero because $(h_v-\beta h_{v-1}-(1-\alpha)(1-\beta)h_{v-2})$ $=0$ by defining
relations of Fibonacci sequences. By replacing $h_1$ and $h_2$, the theorem follows.
\end{proof}

\begin{corollary}The probability mass function of $V(2)$ satisfies the recursive scheme
\begin{equation*}\label{eq: 1.1}
\begin{split}
P(V(2)=v)=\beta P(V(2)=v-1)+(1-\alpha)(1-\beta)P(V(2)=v-2),\ v> 3,
\end{split}.
\end{equation*}
with initial conditions $P(V(2)=0)=P(V(2)=1)=0$, $P(V(2)=2)=p\alpha$ and $P(V(2)=3)=\alpha q(1+\beta)$.
\end{corollary}


Following the similar steps as in the case $k=2$, the following theorems can be easily verified.

\vskip 5pt
\begin{theorem}
The p.g.f. of V(3) is given by
\begin{gather*}
G_{V(3)}(s)=\frac{p \alpha^2 s^3+(q-\beta)\alpha^2s^4}{1-\beta
s-(1-\alpha)(1-\beta)s^2-\alpha(1-\alpha)(1-\beta)s^3}.
\end{gather*}
\end{theorem}
\vskip 5pt

\begin{corollary}The probability mass function of $V(3)$ satisfies the recursive scheme
\begin{equation*}\label{eq: 1.1}
\begin{split}
P(V(3)=v)=&\beta P(V(3)=v-1)+(1-\alpha)(1-\beta)P(V(3)=v-2)\\&+\alpha(1-\alpha)(1-\beta)P(V(3)=v-3),\ v>4,
\end{split}
\end{equation*}
with initial conditions $P(V(3)=0)=P(V(3)=1)=P(V(3)=2)=0$, $P(V(3)=3)=p\alpha^{2}$ and $P(V(3)=4)=q(1-\beta)\alpha^{2}$.
\end{corollary}


\begin{theorem}
The p.g.f. of V(k) is given by
\begin{equation*}\label{pgf:v(k):markov}
\begin{split}
G_{V(k)}(s)&=\frac{\alpha^{k-1}s^k\Big\{p+(q-\beta)s\Big\}}{1-\beta
s-\sum_{i=2}^{k}\alpha^{i-2}(1-\alpha)(1-\beta)s^i}\\
&=\frac{\alpha^{k-1}s^k\Big\{p+(q-\beta)s\Big\}(1-\alpha s)}{1-(\alpha+\beta
)s-(1-\alpha-\beta)s^{2}-\alpha^{k-1}(1-\alpha)(1-\beta)s^{k+1}}.
\end{split}
\end{equation*}
\end{theorem}

\begin{corollary}The probability mass function of $V(k)$ satisfies the recursive scheme
\begin{equation*}\label{eq: 1.1}
\begin{split}
P(V(k)=v)=&(\alpha+\beta) P(V(k)=v-1)-(1-\alpha-\beta)P(V(k)=v-2)\\
&-\alpha^{k-1}(1-\alpha)(1-\beta)P(V(k)=v-k-1),\ v>k+2,
\end{split}
\end{equation*}
with initial conditions
\begin{equation*}
\begin{split}
P(V(k)=v)=\left\{
  \begin{array}{ll}
    0  & \text{if $0\leq v<k$,} \\
    p\alpha^{k-1} & \text{if $v=k$,} \\
    q\alpha^{k-1}(1-\beta) & \text{if $v=k+1$,} \\
    \alpha^{k-1}(1-\beta)\left\{\beta q+p(1-\alpha)\right\} & \text{if $v=k+2$.}\\
  \end{array}
\right.
\end{split}
\end{equation*}
\end{corollary}

From Theorem \ref{pgf:v(k):markov} we get the different type recursive scheme.
\begin{remark}
The probability mass function of $V(k)$ satisfies the recursive scheme
\begin{equation*}
\begin{split}
P(V(k)=v)=&\beta P(V(k)=v-1)\\
&+(1-\alpha)(1-\beta)\sum_{i=2}^{k}\alpha^{i-2}P(V(k)=v-i), \ v> k+1,
\end{split}
\end{equation*}
with initial conditions
\begin{equation*}
\begin{split}
P(V(k)=v)=\left\{
  \begin{array}{ll}
    0,  &\ \text{if}\ 0\leq v<k, \\
    p\alpha^{k-1}, &\ \text{if}\ v=k, \\
    q\alpha^{k-1}(1-\beta), &\ \text{if}\ v=k+1.\\
  \end{array}
\right.
\end{split}
\end{equation*}
\end{remark}


\subsection{The dual relationship between the probability of waiting time of runs and the longest runs}

 Another closely related random variable is the length $L_{n}$ of the longest success run in $n$ trials. Since $V(k)\leq n$ if and only if $L_{n} \geq k$, we have that the relationship between the probability distribution
functions of $V(k)$ and $L_{n}$ is given by the identity
\begin{equation}\label{recu:longest}
\begin{split}
P [V(k) \leq n] = P [L_{n} \geq k].
\end{split}
\end{equation}
\citet{balakrishnan2002runs} derived the relationship between $L_{n}$ and $V(k)$ as follows
\begin{equation}\label{rela:longandgeom}
\begin{split}
P(L_{n}\geq k)&=F(n)=P(V(k)\leq n)\\
P(L_{n}\leq k-1)&=1-F(n)\\
P(L_{n}= k)&=P(V(k)\leq n)-P(V(k+1)\leq n),\ 1\leq k\leq n\\
P(V(k)=x)&=qp^{k}P(L_{x-k-1}< k),\ x\geq k+2\\
P(L_{n} <k)&=\frac{P(V(k)=n+k+1)}{qp^{k}},\ n\geq 1.
\end{split}
\end{equation}

Hence, via Eq. (\ref{rela:longandgeom}) it is offered an alternative way of obtaining exact distribution of $L_{n}$ through formulae established for the run enumerative RV $V(k)$ and vice versa.

\begin{theorem}\label{gene:longest}
The generating function of the longest run probabilities $P(L_{n} = k)$ will be given by
\begin{equation}\label{genelongest}
\begin{split}
\sum_{n=0}^{\infty} P(L_{n}=k)z^{n}=\frac{1}{1-z}\left[G_{V(k)}(z)-G_{V(k+1)}(z)\right]
\end{split}
\end{equation}
\end{theorem}

\begin{proof}
For $k\geq 1$, in view of (\ref{recu:longest}), we have,
\begin{equation}\label{eq: 3.1}
\begin{split}
P(L_{n}=k)&=P(L_{n} \geq k)-P (L_{n} \geq k+1)=P(V(k) \leq n)-P(V(k+1) \leq n)\\
&=\sum_{j=1}^{n}P(V(k)=j)-\sum_{j=1}^{n}P(V(k+1)=j),\ n\geq 1,
\end{split}
\end{equation}
and, therefore,
\begin{equation}\label{eq: 3.1}
\begin{split}
\sum_{n=1}^{\infty} P(L_{n}=k)z^{n}=\sum_{n=1}^{\infty} \sum_{j=1}^{n}P(V(k)=j)z^{n}-\sum_{n=1}^{\infty}\sum_{j=1}^{n}P(V(k+1)=j)z^{n}.
\end{split}
\end{equation}
Interchanging the orders of summation in the RHS, we have,
\begin{equation*}\label{eq: 3.1}
\begin{split}
\sum_{n=1}^{\infty} \sum_{j=1}^{n}P(V(k)=j)z^{n}=&\frac{1}{1-z}\sum_{j=1}^{\infty}P(V(k)=j)z^{j}=\frac{1}{1-z}G_{V(k)}(z),\\
\sum_{n=1}^{\infty}\sum_{j=1}^{n}P(V(k+1)=j)z^{n}=&\frac{1}{1-z}\sum_{j=1}^{\infty}P(V(k+1)=j)z^{j}=\frac{1}{1-z}G_{V(k+1)}(z).
\end{split}
\end{equation*}
We have the formula in (\ref{genelongest}) for $x\geq 1$, using this results. For the special case $k=0$, we first notice that
\begin{equation*}\label{eq: 3.1}
\begin{split}
P(L_{n}=0)&=1-P(V(1) \leq n)=1-\sum_{j=1}^{n}P(V(1)=j),\ n\geq 1,
\end{split}
\end{equation*}
which readily yields
\begin{equation*}\label{eq: 3.1}
\begin{split}
\sum_{n=1}^{\infty}P(L_{n}=0)z^{n}=-1+\frac{1}{1-z}\left[1-G_{V(1)}(z)+P(V(1)=0)\right].
\end{split}
\end{equation*}

\subsubsection{I.i.d. trials}

\end{proof}
Applying of Theorem \ref{gene:longest} could be deduce the probability generating function of the longest run probabilities $P(L_{n} = k)$ as follows
\begin{equation}\label{erre}
\begin{split}
\sum_{n=0}^{\infty} P(L_{n}=k)z^{n}=\frac{P(z)}{Q(z)},
\end{split}
\end{equation}
where
\begin{equation*}\label{gdff}
\begin{split}
P(z)=&p^{k}z^{k}+p^{k}(p^{2k+2}q-2p-1)z^{k+1}+p^{k+1}(2+p-p^{2k+1}q)z^{k+2}-p^{k+2}z^{k+3}\\
&-p^{2k+1}qz^{2k+2}+p^{2k+2}qz^{2k+3},\\
Q(z)=&1+(p^{2k+2}q-3)z+(3-2p^{2k+2}q)z^{2}+(p^{2k+2}q-1)z^{3}+p^{k}qz^{k+1}\\
&+p^{k}q(p^{2k+2}q-2)z^{k+2}+p^{k}q(1-p^{2k+2}q)z^{k+3}.
\end{split}
\end{equation*}

From the last expression for the generating function we have that
\begin{equation}\label{genelongest:pgf}
\begin{split}
P(z)=Q(z) \left(\sum_{n=0}^{\infty} P(L_{n}=k)z^{n}\right)
\end{split}
\end{equation}
If we collect the coefficients of $z^{n}$ $(n = 0, 1,...)$ in both sides of equality, we can
define a set of recurrence relations for the longest run probabilities $P(L_{n} = k)$.\\
Thus, by appropriately utilizing Eq. (\ref{genelongest:pgf}) we conclude with the following recurrent
for the longest run probabilities $P(L_{n} = k)$. Specifically, with the aid of this scheme we can numerical
evaluate the probabilities
\begin{equation*}\label{genelongest:pdf}
\begin{split}
f(n)=P(L_{n} = k)
\end{split}
\end{equation*}
in a more effective and fast way, using the following recursive scheme
\begin{equation*}\label{gdff}
\begin{split}
f(n)=&(3-p^{2k+2}q)f(n-1)+(2p^{2k+2}q-3)f(n-2)+(1-p^{2k+2}q)f(n-3)\\
&-p^{k}qf(n-k-1)+p^{k}q(2-p^{2k+2}q)f(n-k-2)\\
&+p^{k}q(p^{2k+2}q-1)f(n-k-3)+\beta_{n},
\end{split}
\end{equation*}
where we set $f(n)$ for $n<0$ and with initial conditions
\begin{equation*}
\begin{split}
\beta_{n}=\left\{
  \begin{array}{ll}
    p^{k},  &\ \text{if}\ j=k, \\
    p^{k}(p^{2k+2}q-2p-1), &\ \text{if}\ j=k+1, \\
    p^{k+1}(2+p-p^{2k+1}q), &\ \text{if}\ j=k+2, \\
    -p^{k+2}, &\ \text{if}\ j=k+3, \\
    -p^{2k+1}q, &\ \text{if}\ j=2k+2, \\
    p^{2k+2}q, &\ \text{if}\ j=2k+3, \\
    0, &\ \text{otherwise}.\\
  \end{array}
\right.
\end{split}
\end{equation*}

\subsubsection{Markov dependent trials}

Applying of Theorem \ref{gene:longest} could be deduce the probability generating function of the longest run probabilities $P(L_{n} = k)$ as follows
\begin{equation}\label{erre}
\begin{split}
\sum_{n=0}^{\infty} P(L_{n}=k)z^{n}=\frac{P(z)}{Q(z)},
\end{split}
\end{equation}
where
\begin{equation*}\label{gdff}
\begin{split}
P(z)=&p\alpha^{k-1}z^{k}+\alpha^{k-1}\left\{1-\beta-p(3\alpha+\beta)\right\}z^{k+1}\\
&+\alpha^{k-1}\left\{(1-\beta)(1-3\alpha-\beta)-p(1-\alpha-3\alpha^{2}-\beta-2\alpha\beta)\right\}z^{k+2}\\
&+\alpha^{k}\left[(1-\beta)(2-3\alpha-2\beta)+p\left\{\alpha^{2}-2(1-\beta)+\alpha(2+\beta)\right\}\right]z^{k+3}\\
&+\alpha^{k+1}(1-p-\beta)(1-\alpha-\beta)z^{k+4},\\
Q(z)=&1-2(\alpha+\beta)z+\left\{(\alpha+1)^{2}+(\beta+1)^{2}+2\alpha\beta-4\right\}z^{2}\\
&+2(1-\alpha-\beta)(\alpha+\beta)z^{3}+(1-\alpha-\beta)^{2}z^{4}\\
&+\alpha^{k-1}(1-\alpha)(1-\beta)z^{k+1}-\alpha^{k-1}\beta(1-\alpha)(1-\beta)z^{k+2}\\
&-\alpha^{k-1}(1-\alpha)(1-\beta)\left\{1-(1-\alpha)(1-\beta)\beta)\right\}z^{k+3}\\
&-\alpha^{k}(1-\alpha)(1-\beta)(1-\alpha-\beta)z^{k+4}+\alpha^{2k-1}(1-\alpha)^{2}(1-\beta)^{2}z^{2k+3}
\end{split}
\end{equation*}

From the last expression for the generating function we have that
\begin{equation}\label{genelongest:pgf}
\begin{split}
P(z)=Q(z) \left(\sum_{n=0}^{\infty} P(L_{n}=k)z^{n}\right)
\end{split}
\end{equation}
If we collect the coefficients of $z^{n}$ $(n = 0, 1,...)$ in both sides of equality, we can
define a set of recurrence relations for the longest run probabilities $P(L_{n} = k)$.\\
Thus, by appropriately utilizing Eq. (\ref{genelongest:pgf}) we conclude with the following recurrent
for the longest run probabilities $P(L_{n} = k)$. Specifically, with the aid of this scheme we can numerical
evaluate the probabilities
\begin{equation*}\label{genelongest:pdf}
\begin{split}
f(n)=P(L_{n} = k)
\end{split}
\end{equation*}
in a more effective and fast way, using the following recursive scheme
\begin{equation*}\label{gdff}
\begin{split}
f(n)=&2(\alpha+\beta)f(n-1)+\{4-(\alpha+1)^{2}-(\beta+1)^{2}-2\alpha\beta)f(n-2)\\
&-2(1-\alpha-\beta)(\alpha+\beta)f(n-3)-(1-\alpha-\beta)^{2}f(n-4)\\
&-\alpha^{k-1}(1-\alpha)(1-\beta)f(n-k-1)+\alpha^{k-1}\beta(1-\alpha)(1-\beta)f(n-k-2)\\
&+\alpha^{k-1}(1-\alpha)(1-\beta)\{1-(1-\alpha)(1-\beta)\beta f(n-k-3)\\
&+\alpha^{k}(1-\alpha)(1-\beta)\{1-\alpha-\beta)f(n-k-4)\\
&+\alpha^{2k-1}(1-\alpha)^{2}(1-\beta)^{2}f(n-2k-3)+\beta_{n},
\end{split}
\end{equation*}
where we set $f(n)$ for $n<0$ and with initial conditions
\begin{equation*}
\begin{split}
\beta_{n}=\left\{
  \begin{array}{ll}
    p\alpha^{k-1},  &\ \text{if}\ j=k, \\
    \alpha^{k-1}\{1-\beta-p(3\alpha+\beta)\}, &\ \text{if}\ j=k+1, \\
    \alpha^{k-1}\{(1-\beta)(1-3\alpha-\beta)-p(1-\alpha-3\alpha^{2}-\beta-2\alpha\beta)\}, &\ \text{if}\ j=k+2, \\
    \alpha^{k}\{(1-\beta)(2-3\alpha-2\beta)+p(\alpha^{2}-2\alpha(1-\beta)+\alpha(2+\beta)\}, &\ \text{if}\ j=k+3, \\
    \alpha^{k+1}(1-p-\beta)(1-\alpha-\beta), &\ \text{if}\ j=k+3, \\
    0, &\ \text{otherwise}.\\
  \end{array}
\right.
\end{split}
\end{equation*}

\section{Distribution of waiting time until the $r$-th occurrence of a pattern}
 Let us denote by $T_{r,k}^{(a)}$, $r=1,\ 2,\ ...$, $a=\rom{1},\rom{2},\rom{3}$, the waiting time until the $r$-th occurrence of a pattern and its probability mass function by
\begin{equation}\label{eq: 1.1}
\begin{split}
h_{r}^{(a)}(n)&= \mathbf{P}(T_{r,k}^{(a)}=n),\ n\geq 0,\ a=\rom{1},\rom{2},\rom{3},
\end{split}
\end{equation}
and its single and double probability generating functions of $T_{r,k}^{(a)}$, $(r=1,\ 2,\ ...,\ a=\rom{1},\rom{2},\rom{3})$ will be denoted by $H_{r}^{(a)}(z)$ and $H(z,w)^{(a)}$, respectively; i.e.,
\begin{equation}\label{eq: 1.1}
\begin{split}
H_{r}^{(a)}(z)&=\sum_{n=1}^{\infty}h_{r}^{(a)}(n)z^{n},\ r\geq 0\\
H^{(a)}(z,w)&=\sum_{r=1}^{\infty} H_{r}^{(a)}(z)w^{r}.
\end{split}
\end{equation}
In this section, we derive the p.g.f., the moments and the recursive scheme of $T_{r,k}^{(a)}$, $a=\rom{1},\rom{2},\rom{3}$.

\subsection{I.i.d. trials}
\subsubsection{Non-overlapping scheme}
We can derive probability generating function of $T_{r,k}^{(\rom{1})}$ for $|z| \leq 1$ by definition of negative binomial distribution of order $k$.\\
\begin{equation}\label{eq:type1-pgf-i.i.d.}
\begin{split}
H_{r}^{(\rom{1})}(z)&=\sum_{x=rk}^{\infty} \mathbf{P}(T_{r,k}^{(\rom{1})}=x)z^{x}=\prod_{i=1}^{r} \mathbf{E}(z^{X_{i}})=\left(\frac{p^{k}z^{k}(1-pz)}{1-z+qp^{k}z^{k+1}}\right)^{r}.
\end{split}
\end{equation}

First, we calculate the double generating function.
\begin{proposition}
The double probability generating function $H^{(\rom{1})}(z,w)$ of $T_{r,k}^{(\rom{1})}$ is given by
\begin{equation}\label{eq:3.1}
\begin{split}
H^{(\rom{1})}(z,w)=\frac{1-z+p^{k}q z^{k+1}}{1-z+qp^{k}z^{k+1}-(p z)^{k}w+(p z)^{k+1}w}.
\end{split}
\end{equation}

\end{proposition}

\begin{proof}
\begin{equation}\label{}
\begin{split}
H^{(\rom{1})}(z,w)&=\sum_{r=0}^{\infty}H_{r}^{(\rom{1})}(z)w^{r}=\sum_{r=0}^{\infty}\left(\frac{p^{k}z^{k}(1-pz)}{1-z+qp^{k}z^{k+1}}\right)^{r}w^{r}\\
&=\left(\frac{1-z+p^{k}q z^{k+1}}{1-z+qp^{k}z^{k+1}-(p z)^{k}w+(p z)^{k+1}w}\right).
\end{split}
\end{equation}
\end{proof}

In the following lemma we derive a recursive scheme for the evaluation of $H_{r}(z)$.
\begin{lemma}\label{lem:1}
The probability generating function $H_{r}^{(\rom{1})}(z)$ of the random variable $T_{r,k}^{(\rom{1})}$ is satisfies the recursive scheme
\begin{equation}\label{eq: lem1}
\begin{split}
H_{r}^{(\rom{1})}(z)=\left(\frac{(p z)^{k}(1-pz)}{1-z+qp^{k}z^{k+1}}\right) H_{r-1}^{(\rom{1})}(z)
\end{split}
\end{equation}
with initial conditions $H_{0}^{(\rom{1})}(z)=1$.
\end{lemma}

\begin{proof}
It follows by equating the coefficients of $w^{r}$ on both sides of (\ref{eq:3.1}).
\end{proof}

%
%
%
%
%

We can establish an appropriate recurrent relation for the PMF of the random variable $T_{r,k}^{(\rom{1})}$ using the result established in (\ref{eq: lem1}), is given in the following theorem.


\begin{theorem}\label{thm:type1-recu-i.i.d.}
The probability mass function $h_{r}^{(\rom{1})}(n)$ of the random variable $T_{r,k}^{(\rom{1})}$ satisfies the recursive scheme
\begin{equation}\label{eq: 1.1}
\begin{split}
h_{r}^{(\rom{1})}(n)= h_{r-1}^{(\rom{1})}(n-1)-p^{k} qh_{r}^{(\rom{1})}(n-k-1)+p^{k}h_{r-1}^{(\rom{1})}(n-k)-p^{k+1}h_{r-1}^{(\rom{1})}(n-k-1)
\end{split}
\end{equation}
with initial conditions $h_{0}^{(\rom{1})}(n)=\delta_{n,0}$.
\end{theorem}

\begin{proof}
It suffices to replace $H_{r}^{(\rom{1})}(z)$,in the recursive formulae given in Lemma \ref{lem:1} by the power series
\begin{equation}\label{eq: 3.1}
\begin{split}
H_{r}^{(\rom{1})}(z)=\sum_{n=0}^{\infty}P(T_{r,k}^{(\rom{1})}=n)z^{n}=\sum_{n=0}^{\infty}h_{r}^{(\rom{1})}(n)z^{n},
\end{split}
\end{equation}
and then equating the coefficients of $z^{n}$ on both sides of the resulting identities.
\end{proof}

If we have recursive scheme for the probability mass function of $T_{r,k}^{(\rom{1})}$, it is not difficult to derive a recursive scheme for the tail probabilities of $T_{r,k}^{(\rom{1})}$. More specifically we have the following result.

\begin{corollary}The tail probability function $\overline{h}_{r}{(\rom{1})}(n)=P(T_{r,k}^{(\rom{1})}>n)$ of $T_{r,k}^{(\rom{1})}$ satisfies the recurrence relation.
\begin{equation*}\label{eq: 1.1}
\begin{split}
\overline{h}_{r}^{(\rom{1})}(n)=&2\overline{h}_{r-1}^{(\rom{1})}(n-1)-\overline{h}_{r-1}^{(\rom{1})}(n-2)+p^{k}q\left\{\overline{h}_{r}^{(\rom{1})}(n-k-2)-\overline{h}_{r}^{(\rom{1})}(n-k-1)\right\}\\
&-p^{k}\left\{\overline{h}_{r-1}^{(\rom{1})}(n-k-1)-\overline{h}_{r-1}^{(\rom{1})}(n-k)\right\}\\
&+p^{k+1}\left\{\overline{h}_{r-1}^{(\rom{1})}(n-k-2)-\overline{h}_{r-1}^{(\rom{1})}(n-k-1)\right\}.
\end{split}
\end{equation*}
\end{corollary}

\begin{proof}
It can be easily derive either from the outcome of Theorem \ref{thm:type1-recu-i.i.d.} upon replacing $h_{r}(n)$ by $\overline{h}_{r}(n-1)-\overline{h}_{r}(n)$ or from the outcome equation (\ref{eq:type1-pgf-i.i.d.}) on observing that the generating function of $\overline{h}_{r}(n)$ may be expressed in terms of $H_{r}(z)$ by
\begin{equation*}\label{eq: 1.1}
\begin{split}
\sum_{n=0}^{\infty}\overline{h}_{r}(n)z^{n}=\frac{1-H_{r}(z)}{1-z}.
\end{split}
\end{equation*}
\end{proof}

Koutras (1997) derived formulae expressing the generating functions of moments of $T_{r,k}$ and $(T_{r,k})^{2}$, respectively, $\sum_{r=0}^{\infty}E[T_{r,k}^{(\rom{1})}]w^{r}=\left[\frac{\partial}{\partial z}H(z,w)\right]_{z=1}$ and $\sum_{r=0}^{\infty}(T_{r,k})^{2}w^{r}=\frac{\partial}{\partial z}\left[z\frac{\partial}{\partial z}H(z,w)\right]_{z=1}$. Using this result, we derive some formulae expressing the generating functions of the first two moments of $T_{r,k}^{(\rom{1})}$ and $(T_{r,k}^{(\rom{1})})^{2}$ by means of the double generating function $H^{(\rom{1})}(z,w)$.

\begin{theorem}
The generating function of the means $E[T_{r,k}^{(\rom{1})}]$ is given by
\begin{equation}\label{eq:type1-first-mean}
\begin{split}
\sum_{r=0}^{\infty}E[T_{r,k}^{(\rom{1})}]w^{r}=\frac{(1-p^{k})w}{p^{k}(1-w)^{2}}.
\end{split}
\end{equation}
\end{theorem}

We can establish an appropriate recurrent relation for the $\mu_{r}^{(\rom{1})}=E[T_{r,k}^{(\rom{1})}]$ of the random variable $T_{r,k}^{(\rom{1})}$ using the result established in (\ref{eq:type1-first-mean}), is given in the following result.

\begin{corollary}The $\mu_{r}^{(\rom{1})}=E[T_{r,k}^{(\rom{1})}]$, $r=1,\ 2, ...$ of the random variable $T_{r}^{(\rom{1})}$ satisfies the recurrence relation.
\begin{equation}\label{eq:recu-type1-mean-first}
\begin{split}
\mu_{r}^{(\rom{1})}=2\mu_{r-1}^{(\rom{1})}-\mu_{r-2}^{(\rom{1})},\ r\geq 3.
\end{split}
\end{equation}
with initial conditions $\mu_{0}^{(\rom{1})}=0$, $\mu_{1}^{(\rom{1})}=\frac{1-p^{k}}{p^{k}}$.
\end{corollary}

\begin{proof}
It follows by equating the coefficients of $w^{r}$ on both sides of (\ref{eq:type1-first-mean}).
\end{proof}


\begin{theorem}
The generating function of the means $E[(T_{r,k}^{(\rom{1})})^{2}]$ is given by
\begin{equation}\label{eq:type1-second-mean}
\begin{split}
\sum_{r=0}^{\infty}E[(T_{r,k}^{(\rom{1})})^{2}]w^{r}=\frac{a_{1}w^{2}+a_{2}w}{p^{2k}q^{3}(1-w)^{3}},
\end{split}
\end{equation}
where
\begin{equation}\label{eq: 1.1}
\begin{split}
a_{1}&=p^{k}q\{(1+p)(p^{k}-1)+2qk\},\ a_{2}=q\{(1-p^{k})(2-p^{k}q)-2kp^{k}q\}
\end{split}
\end{equation}
\end{theorem}

\begin{corollary}The $U_{r}^{(\rom{1})}=E[(T_{r,k}^{(\rom{1})})^{2}]$, $r=1,\ 2, ...$ of the random variable $(T_{r,k}^{(\rom{1})})^{2}$ satisfies the recurrence relation.
\begin{equation}\label{eq:recu-type1-mean-first}
\begin{split}
U_{r}^{(\rom{1})}=3U_{r-1}^{(\rom{1})}-3U_{r-2}^{(\rom{1})}+U_{r-3}^{(\rom{1})},\ r\geq 4.
\end{split}
\end{equation}
with initial conditions
\begin{equation*}\label{eq: 1.1}
\begin{split}
U_{0}^{(\rom{1})}&=0,\ U_{1}^{(\rom{1})}=\frac{q\{(1-p^{k})(2-p^{k}q)-2kp^{k}q\}}{p^{2k}q^{3}},\\
U_{2}^{(\rom{1})}&=\frac{2p^{2k}(1+q)+2kp^{k}(p-2qk-5)+6}{p^{2k}q^{2}}.
\end{split}
\end{equation*}

\end{corollary}

\begin{proof}
It follows by equating the coefficients of $w^{r}$ on both sides of (\ref{eq:type1-second-mean}).
\end{proof}


\subsubsection{Exceed a threashold scheme}
\citet{balakrishnan2002runs} derived the relationship between $H_{r}^{(\rom{2})}(z)$ and $G_{V(k)}(z)$. We already derived $G_{V(k)}(s)$ in Section 2. We derive $H_{r}^{(\rom{2})}(z)$ using this result.

\begin{equation}\label{eq: 1.1}
\begin{split}
H_{r}^{(\rom{2})}(z)&=\mathbf{E}(z^{T_{r,k}^{(\rom{2})}})=\sum_{n=0}^{\infty}h_{r}^{(\rom{2})}(n)z^{n}=\left[G_{V(k)}(z) \frac{qz}{1-pz}\right]^{r-1}G_{V(k)}(z)\\
&=\left[G_{V(k)}(z) \right]^{r} \left[\frac{qz}{1-pz}\right]^{r-1},\\
\end{split}
\end{equation}
where
\begin{equation}\label{eq: 1.1}
\begin{split}
G_{V(k)}(z)=\frac{(p z)^{k}(1-pz)}{1-z+qp^{k}z^{k+1}}.
\end{split}
\end{equation}
First, we calculate the double generating function.
\begin{proposition}
The double probability generating function $H(z,w)$ of $T_{r,k}^{(\rom{2})}$ is given by
\begin{equation}\label{eq:3.1}
\begin{split}
H^{(\rom{2})}(z,w)&=\frac{1-z+p^{k}q z^{k+1}+p^{k}z^{k}(1-z)w}{1-z+p^{k}q z^{k+1}-p^{k}q z^{k+1}w.}
\end{split}
\end{equation}
\end{proposition}


In the following lemma we derive a recursive scheme for the evaluation of $H_{r}(z)$.
\begin{lemma}\label{lem:1}
The probability generating function $H^{(\rom{2})}_{r}(z)$ of the random variable $T_{r,k}^{(\rom{2})}$ is satisfies the recursive scheme
\begin{equation}\label{eq: 1.1}
\begin{split}
H_{r}^{(\rom{2})}(z)=\frac{p^{k}qz^{k+1}}{1-z+p^{k}qz^{k+1}} H_{r-1}^{(\rom{2})}(z),\ \text{for}\ r>1,
\end{split}
\end{equation}
with initial condition $H_{0}^{(\rom{2})}(z)=1$ and $H_{1}^{(\rom{2})}(z)=\frac{p^{k}z^{k}(1-pz)}{1-z+p^{k}qz^{k+1}}$.
\end{lemma}


%
%
%
%
%

An efficient recursive scheme for the evaluation of the probability mass function of $T_{r,k}^{(\rom{2})}$, ensuing from the result established in \ref{lem:1}, is given in the following theorem.

\begin{theorem}
The probability mass function $h_{r}^{(\rom{2})}(n)$ of the random variable $T_{r,k}^{(\rom{2})}$ satisfies the recursive scheme
\begin{equation}\label{eq: 1.1}
\begin{split}
h_{r}^{(\rom{2})}(n)= h_{r}^{(\rom{2})}(n-1)-p^{k} qh_{r}^{(\rom{2})}(n-k-1)+p^{k}qh_{r-1}^{(\rom{2})}(n-k-1)
\end{split}
\end{equation}
with initial conditions $h_{0}^{(\rom{2})}(n)=\delta_{n,0}$ and
\begin{equation*}
\begin{split}
h_{1}^{(\rom{2})}(n)=\left\{
  \begin{array}{ll}
    0  & \text{if $0\leq n<k$,} \\
    p^{k} & \text{if $n=k$,} \\
    p^{k}q & \text{if $n=k+1$,} \\
    h_{1}^{(\rom{2})}(n)=h_{1}^{(\rom{2})}(n-1)-p^{k}qh_{1}^{(\rom{2})}(n-k-1) & \text{if}\ n>k+1.\\
  \end{array}
\right.
\end{split}
\end{equation*}
\end{theorem}

%

\begin{corollary}The tail probability function $\overline{h}_{r}^{(\rom{2})}(n)=P(T_{r,k}^{(\rom{2})}>n)$ of $T_{r,k}^{(\rom{2})}$ satisfies the recurrence relation.
\begin{equation*}\label{eq: 1.1}
\begin{split}
\overline{h}_{r}^{(\rom{2})}(n)=&2\overline{h}_{r}^{(\rom{2})}(n-1)-\overline{h}_{r}^{(\rom{2})}(n-2)+p^{k}q\left\{\overline{h}_{r}^{(\rom{2})}(n-k-2)-\overline{h}_{r}^{(\rom{2})}(n-k-1)\right\}\\
&-p^{k}q\left\{\overline{h}_{r}^{(\rom{2})}(n-k-2)-\overline{h}_{r}^{(\rom{2})}(n-k-1)\right\}.
\end{split}
\end{equation*}
\end{corollary}

Let us give some formulae expressing the generating functions of the first two moments of $T_{r,k}^{(\rom{2})}$ and $(T_{r,k}^{(\rom{2})})^{2}$ by means of the double generating function $H(z,w)$.

\begin{theorem}
The generating function of the means $E[T_{r,k}^{(\rom{2})}]$ is given by
\begin{equation}\label{eq: 1.1}
\begin{split}
\sum_{r=0}^{\infty}E[T_{r,k}^{(\rom{2})}]w^{r}=\frac{(1-p^{k})w+p^{k}w^{2}}{p^{k}q(1-w)^{2}}.
\end{split}
\end{equation}
\end{theorem}

\begin{corollary}The $\mu_{r}^{(\rom{2})}=E[T_{r,k}^{(\rom{2})}]$, $r=1,\ 2, ...$ of the random variable $T_{r}^{(\rom{2})}$ satisfies the recurrence relation.
\begin{equation}\label{eq:recu-type1-mean-first}
\begin{split}
\mu_{r}^{(\rom{2})}=2\mu_{r-1}^{(\rom{2})}-\mu_{r-2}^{(\rom{2})},\ \text{for}\ r>2,
\end{split}
\end{equation}
with initial conditions $\mu_{0}^{(\rom{2})}=0$, $\mu_{1}^{(\rom{2})}=\frac{1-p^{k}}{p^{k}q}$ and $\mu_{2}^{(\rom{2})}=\frac{2-p^{k}}{p^{k}q}$.
\end{corollary}


\begin{theorem}
The generating function of the means $E[(T_{r,k}^{(\rom{2})})^{2}]$ is given by
\begin{equation*}\label{eq: 1.1}
\begin{split}
\sum_{r=0}^{\infty}E[(T_{r,k}^{(\rom{2})})^{2}]w^{r}=\frac{a_{1}w+a_{2}w^{2}+a_{3}w^{3}}{q^{2}p^{2k}(1-w)^{3}},
\end{split}
\end{equation*}
where
\begin{equation*}\label{eq: 1.1}
\begin{split}
a_{1}&=2+p^{k}\{p-3+q(p^{k}-2k)\},\ a_{2}=-p^{k}\{p-3+2q(p^{k}-k)\},\ a_{3}=p^{2k}q.
\end{split}
\end{equation*}
\end{theorem}


\begin{corollary}The $U_{r}^{(\rom{2})}=E[(T_{r,k}^{(\rom{2})})^{2}]$, $r=1,\ 2, ...$ of the random variable $(T_{r,k}^{(\rom{2})})^{2}$ satisfies the recurrence relation.
\begin{equation}\label{eq:recu-type1-mean-first}
\begin{split}
U_{r}^{(\rom{2})}=3U_{r-1}^{(\rom{2})}-3U_{r-2}^{(\rom{2})}+U_{r-3}^{(\rom{2})},\ r\geq 4.
\end{split}
\end{equation}
with initial conditions
\begin{equation}\label{eq: 3.1}
\begin{split}
U_{0}^{(\rom{2})}=&0,\ U_{1}^{(\rom{2})}=\frac{2+p^{k}\{p-3+q(p^{k}-2k)\}}{p^{2k}q^{2}},\\ U_{2}^{(\rom{2})}=&\frac{6+p^{k}\{2(p-3)+q(p^{k}-4k)\}}{p^{2k}q^{2}},\ U_{3}^{(\rom{2})}=\frac{12+p^{k}\{3(p-3)+q(p^{k}-6k)\}}{p^{2k}q^{2}}.
\end{split}
\end{equation}
\end{corollary}

\subsubsection{Overlapping scheme}

\citet{balakrishnan2002runs} derived the relationship between $H_{r}^{(\rom{3})}(z)$ and $G_{V(k)}(z)$. We already derived $G_{V(k)}(s)$ in Section 2. We derive $H_{r}^{(\rom{3})}(z)$ using this result.

\begin{equation}\label{eq: 1.1}
\begin{split}
H_{r}^{(\rom{3})}(z)&=\mathbf{E}(z^{T_{r,k}^{(\rom{3})}})=\sum_{n=0}^{\infty}h_{r}^{(\rom{3})}(n)z^{n}=\left\{pz+(qz)G_{V(k)}(z) \right\}^{r-1}G_{V(k)}(z)\\
&=\frac{(1-pz)(pz)^{k+r-1}(1-z+qp^{k-1}z^{k})^{r-1}}{(1-z+qp^{k}z^{k+1})^{r}},\\
\end{split}
\end{equation}
where
\begin{equation*}\label{eq: 1.1}
\begin{split}
G_{V(k)}(z)=\frac{(p z)^{k}(1-pz)}{1-z+qp^{k}z^{k+1}}.
\end{split}
\end{equation*}
First, we calculate the double generating function.
\begin{proposition}
The double probability generating function $H(z,w)$ of $T_{r,k}^{(\rom{3})}$ is given by
\begin{equation}\label{eq:3.1}
\begin{split}
H^{(\rom{3})}(z,w)=\frac{1-z+p^{k}qz^{k+1}+w(1-z)pz(p^{k-1}z^{k-1}-1)}{1-z+p^{k}qz^{k+1}-wpz(1-z+p^{k-1}qz^{k})}.
\end{split}
\end{equation}

\end{proposition}


In the following lemma we derive a recursive scheme for the evaluation of $H_{r}(z)$.
\begin{lemma}\label{lem:1}
The probability generating function $H^{(\rom{3})}_{r}(z)$ of the random variable $T_{r,k}^{(\rom{3})}$ is satisfies the recursive scheme
\begin{equation}\label{eq: 1.1}
\begin{split}
H_{r}^{(\rom{3})}(z)=\left[\frac{pz(1-z+p^{k-1}qz^{k})}{1-z+p^{k}qz^{k+1}}\right] H_{r-1}^{(\rom{3})}(z)
\end{split}
\end{equation}
\end{lemma}


%
%
%
%
%

An efficient recursive scheme for the evaluation of the probability mass function of $T_{r,k}^{(\rom{3})}$, ensuing from the result established in \ref{lem:1}, is given in the following theorem.

\begin{theorem}
The probability mass function $h_{r}^{(\rom{3})}(n)$ of the random variable $T_{r,k}^{(\rom{3})}$ satisfies the recursive scheme
\begin{equation}\label{eq: 1.1}
\begin{split}
h_{r}^{(\rom{3})}(n)=& h_{r}^{(\rom{3})}(n-1)-p^{k} qh_{r}^{(\rom{3})}(n-k-1)+ph_{r-1}^{(\rom{3})}(n-1)\\
&-ph_{r-1}^{(\rom{3})}(n-2)+p^{k}qh_{r-1}^{(\rom{3})}(n-k-1).
\end{split}
\end{equation}
\end{theorem}


\begin{corollary}The tail probability function $\overline{h}_{r}{(\rom{3})}(n)=P(T_{r,k}{(\rom{3})}>n)$ of $T_{r,k}{(\rom{3})}$ satisfies the recurrence relation.
\begin{equation*}\label{eq: 1.1}
\begin{split}
\overline{h}_{r}^{(\rom{3})}(n)=&2\overline{h}_{r}^{(\rom{3})}(n-1)-\overline{h}_{r}^{(\rom{3})}(n-2)\\
&+p^{k}q\left\{\overline{h}_{r}^{(\rom{3})}(n-k-2)-\overline{h}_{r}^{(\rom{3})}(n-k-1)\right\}\\
&+p\overline{h}_{r-1}^{(\rom{3})}(n-1)-2\overline{h}_{r-1}^{(\rom{3})}(n-2)+p\overline{h}_{r-1}^{(\rom{3})}(n-3)\\
&-p^{k}q\left\{\overline{h}_{r-1}^{(\rom{3})}(n-k-2)-\overline{h}_{r-1}^{(\rom{3})}(n-k-1)\right\}.
\end{split}
\end{equation*}
\end{corollary}

Let us give some formulae expressing the generating functions of the first two moments of $T_{r,k}^{(\rom{3})}$ and $(T_{r,k}^{(\rom{3})})^{2}$ by means of the double generating function $H(z,w)$.

\begin{theorem}
The generating function of the means $E[T_{r,k}^{(\rom{3})}]$ is given by
\begin{equation}\label{eq: 1.1}
\begin{split}
\sum_{r=0}^{\infty}E[T_{r,k}^{(\rom{3})}]w^{r}=\frac{(p^{k}-p)w^{2}+(1-p^{k})w}{p^{k}q(1-w)^{2}}.
\end{split}
\end{equation}
\end{theorem}

\begin{corollary}The $\mu_{r}^{(\rom{3})}=E[T_{r,k}^{(\rom{3})}]$, $r=1,\ 2, ...$ of the random variable $T_{r}^{(\rom{3})}$ satisfies the recurrence relation.
\begin{equation}\label{eq:recu-type1-mean-first}
\begin{split}
\mu_{r}^{(\rom{3})}=2\mu_{r-1}^{(\rom{3})}-\mu_{r-2}^{(\rom{3})},\ r\geq 3.
\end{split}
\end{equation}
with initial conditions
\begin{equation}\label{eq: 3.1}
\begin{split}
\mu_{0}^{(\rom{3})}=&0,\ \mu_{1}^{(\rom{3})}=\frac{1-p^{k}}{p^{k}q},\ \mu_{2}^{(\rom{3})}=\frac{q+1-p^{k}}{p^{k}q}.
\end{split}
\end{equation}
\end{corollary}


\begin{theorem}
The generating function of the means $E[(T_{r,k}^{(\rom{3})})^{2}]$ is given by
\begin{equation*}\label{eq: 1.1}
\begin{split}
\sum_{r=0}^{\infty}E[(T_{r,k}^{(\rom{3})})^{2}]w^{r}=\frac{a_{1}w+a_{2}w^{2}+a_{3}w^{3}}{p^{2k}q^{2}(1-w)^{3}},
\end{split}
\end{equation*}
where
\begin{equation*}\label{eq: 1.1}
\begin{split}
a_{1}&=\left[2+p^{k}\{p-3+q(p^{k}-2k)\}\right],\ a_{2}=\left[p^{k}\{3+p^{2}+2kq(1+p)\}-4p-2p^{2k}q\right],\\
a_{3}&=\{2p^{2}+p^{2k}q-(1+p+2qk)p^{k+1}\}.
\end{split}
\end{equation*}
\end{theorem}


\begin{corollary}The $U_{r}^{(\rom{2})}=E[(T_{r,k}^{(\rom{2})})^{2}]$, $r=1,\ 2, ...$ of the random variable $(T_{r,k}^{(\rom{2})})^{2}$ satisfies the recurrence relation.
\begin{equation}\label{eq:recu-type1-mean-first}
\begin{split}
U_{r}^{(\rom{3})}=3U_{r-1}^{(\rom{3})}-3U_{r-2}^{(\rom{3})}+U_{r-3}^{(\rom{3})},\ r\geq 4.
\end{split}
\end{equation}
with initial conditions
\begin{equation}\label{eq: 3.1}
\begin{split}
U_{0}^{(\rom{3})}=&0,\ U_{1}^{(\rom{3})}=\frac{2+p^{k}\{p-3+q(p^{k}-2k)\}}{p^{2k}q^{2}},\\
U_{2}^{(\rom{3})}=&\frac{6-4p+p^{2k}q-p^{k}\{6+2qk(2-p)-p(3+p)\}}{p^{2k}q^{2}},\\
U_{3}^{(\rom{3})}=&\frac{12-12p+2p^{2}+p^{2k}q+p^{k}\{p(2p+5)-2kq(3-2p)-9\}}{p^{2k}q^{2}}.\\
\end{split}
\end{equation}
\end{corollary}

\subsection{Markov dependent trials}
Consider an infinite sequence of two state Markov dependent Bernoulli trials $\{X_i|i=1,2,\dots\}$ with
P($X_i=1$)$=p$, P($X_i=0$)$=1-p=q$, P($X_i=1|X_{i-1}=1$)$=\alpha$ and P($X_i=0|X_{i-1}=0$)$=\beta$ for $i > 1$. Feller (1968) introduced \textit{delayed recurrent event}, supposing that the distribution of success runs of length $k$ be for \textit{delayed recurrent event}, and therefore

\begin{equation}\label{equ:delayed-recu}
\begin{split}
H_{r}^{a}(z)=H^{a}(z)\left[A^{a}(z)\right]^{r-1},\ r\geq 1,\ a=\rom{1},\rom{2},\rom{3},
\end{split}
\end{equation}
where $H^{a}(z)$ and $A^{a}(z)$ are proper p.g.fs. \citet{balakrishnan2002runs} derived the proper p.g.fs. $H^{a}(z)$ and $A^{a}(z)$,
\begin{equation}\label{eq: 1.1}
\begin{split}
H^{(\rom{1})}(z)=&\frac{(\alpha z)^{k-1} \left[p+\{q(1-\beta)-\beta p\}z\right]z}{1-\beta z-\sum_{i=2}^{k}\alpha^{i-2}(1-\alpha)(1-\beta)z^{i}},\\
A^{(\rom{1})}(z)=&\frac{(\alpha z)^{k-1} \left[\alpha+\{(1-\alpha)(1-\beta)-\alpha\beta \}z\right]z}{1-\beta z-\sum_{i=2}^{k}\alpha^{i-2}(1-\alpha)(1-\beta)z^{i}}.
\end{split}
\end{equation}

\begin{equation}\label{eq: 1.1}
\begin{split}
H^{(\rom{2})}(z)=&\frac{(\alpha z)^{k-1} \left[p+\{q(1-\beta)-\beta p\}z\right]z}{1-\beta z-\sum_{i=2}^{k}\alpha^{i-2}(1-\alpha)(1-\beta)z^{i}},\\
A^{(\rom{2})}(z)=&\frac{(1-\alpha)z}{1-\alpha z}\frac{(\alpha z)^{k-1}(1-\beta)z}{1-\beta z-\sum_{i=2}^{k}\alpha^{i-2}(1-\alpha)(1-\beta)z^{i}}.
\end{split}
\end{equation}

\begin{equation}\label{eq: 1.1}
\begin{split}
H^{(\rom{3})}(z)=&\frac{(\alpha z)^{k-1} \left[p+\{q(1-\beta)-\beta p\}z\right]z}{1-\beta z-\sum_{i=2}^{k}\alpha^{i-2}(1-\alpha)(1-\beta)z^{i}},\\
A^{(\rom{3})}(z)=&\left[(\alpha z)+(1-\alpha) z \frac{(1-\beta) z(\alpha z)^{k-1}}{1-\beta z-\sum_{i=2}^{k}\alpha^{i-2}(1-\alpha)(1-\beta)z^{i}}\right].
\end{split}
\end{equation}

\subsubsection{Non-overlapping scheme}

%

\begin{proposition}
The double probability generating function $H^{(\rom{1})}(z,w)$ of $T_{r,k}^{(\rom{1})}$ is given by
\begin{equation}\label{eq:3.1}
\begin{split}
H^{(\rom{1})}(z,w)=\frac{P(z,w)}{Q(z,w)},
\end{split}
\end{equation}
where
\begin{equation*}\label{type1-nega-dpgf-markov}
\begin{split}
P(z,w)=&R(z)+wP(z),\ Q(z,w)=R(z)+wQ(z),\\
R(z)=&1-(\alpha+\beta)z-(1-\alpha-\beta)z^{2}+\alpha^{k-1}(1-\alpha)(1-\beta)z^{k+1},\\
P(z)=&\alpha^{k-1}(p-\alpha)\left\{\alpha z^{k+2}-(\alpha+1)z^{k+1}+z^{k}\right\},\\
Q(z)=&\alpha^{k-1}\left\{\alpha(1-\alpha-\beta)z^{k+2}-(1-\alpha-\alpha^{2}-\beta)z^{k+1}-\alpha z^{k}\right\}.
\end{split}
\end{equation*}
\end{proposition}

\begin{proof}
\begin{equation}\label{}
\begin{split}
H^{(\rom{1})}(z,w)&=\sum_{r=0}^{\infty}H_{r}^{(\rom{1})}(z)w^{r}=1+\sum_{r=1}^{\infty}H^{(\rom{1})}(z)\left[A^{(\rom{1})}(z)\right]^{r-1}w^{r}=\frac{P(z,w)}{Q(z,w)}.
\end{split}
\end{equation}
\end{proof}

In the following lemma we derive a recursive scheme for the evaluation of $H_{r}(z)$.
\begin{lemma}\label{lem:1}
The probability generating function $H_{r}^{(\rom{1})}(z)$ of the random variable $T_{r,k}^{(\rom{1})}$ is satisfies the recursive scheme
\begin{equation}\label{eq: 1.1}
\begin{split}
H_{r}^{(\rom{1})}(z)=\left[\frac{\alpha^{k-1}\left\{\alpha z^{k}-(1-\alpha-\alpha^{2}-\beta)z^{k+1}-\alpha(1-\alpha-\beta)z^{k+2}\right\}}{1-(\alpha+\beta)z-(1-\alpha-\beta)z^{2}+(1-\alpha)(1-\beta)\alpha^{k-1}z^{k+1}}\right] H_{r-1}^{(\rom{1})}(z),\ r>2,
\end{split}
\end{equation}
 with initial conditions $H_{0}^{(\rom{1})}(z)=1$, $H_{1}^{(\rom{1})}(z)=\frac{\alpha^{k-1}\left\{\alpha(\beta-q)z^{k+2}+(q-\beta-\alpha p)z^{k+1}+p z^{k}\right\}}{1-(\alpha+\beta)z-(1-\alpha-\beta)z^{2}+\alpha^{k-1}(1-\alpha)(1-\beta)z^{k+1}.}$
\end{lemma}

\begin{proof}
It follows by equating the coefficients of $w^{r}$ on both sides of (\ref{type1-nega-dpgf-markov}).
\end{proof}

%
%
%
%
%

An efficient recursive scheme for the evaluation of the probability mass function of $T_{r,k}^{(\rom{1})}$, ensuing from the result established in \ref{lem:1}, is given in the following theorem.

\begin{theorem}
The probability mass function $h_{r}^{(\rom{1})}(n)$ of the random variable $T_{r,k}^{(\rom{1})}$ satisfies the recursive scheme
\begin{equation}\label{eq: 1.1}
\begin{split}
h_{r}^{(\rom{1})}(n)=& (\alpha+\beta)h_{r}^{(\rom{1})}(n-1)+(1-\alpha-\beta)h_{r}^{(\rom{1})}(n-2)\\
&-\alpha^{k-1}(1-\alpha)(1-\beta)h_{r}^{(\rom{1})}(n-k-1)\\
&+\alpha^{k}h_{r-1}^{(\rom{1})}(n-k)-(1-\alpha-\alpha^{2}-\beta)\alpha^{k-1}h_{r-1}^{(\rom{1})}(n-k-1)\\
&-(1-\alpha-\beta)\alpha^{k}h_{r-1}^{(\rom{1})}(n-k-2),\ r>1,\ n>k+2,
\end{split}
\end{equation}
with initial conditions $h_{0}^{(\rom{1})}(n)=\delta_{n,0}$ and
\begin{equation*}
\begin{split}
h_{1}^{(\rom{1})}(n)=\left\{
  \begin{array}{ll}
    0  & \text{if $0\leq n<k$,} \\
    p\alpha^{k-1} & \text{if $n=k$,} \\
    q\alpha^{k-1}(1-\beta) & \text{if $n=k+1$,} \\
    \alpha^{k-1}(1-\beta)\left\{\beta q+p(1-\alpha)\right\} & \text{if $n=k+2$.}\\
  \end{array}
\right.
\end{split}
\end{equation*}
\end{theorem}

\begin{proof}
It suffices to replace $H_{r}^{(\rom{1})}(z)$,in the recursive formulae given in Lemma \ref{lem:1} by the power series
\begin{equation}\label{eq: 3.1}
\begin{split}
H_{r}^{(\rom{1})}(z)=\sum_{n=0}^{\infty}P(T_{r,k}^{(\rom{1})}=n)z^{n}=\sum_{n=0}^{\infty}h_{r}^{(\rom{1})}(n)z^{n},
\end{split}
\end{equation}
and then equating the coefficients of $z^{n}$ on both sides of the resulting identities.
\end{proof}

\begin{corollary}The tail probability function $\overline{h}_{r}^{(\rom{1})}(n)=P(T_{r,k}^{(\rom{1})}>n)$ of $T_{r,k}^{(\rom{1})}$ satisfies the recurrence relation.
\begin{equation*}\label{eq: 1.1}
\begin{split}
\overline{h}_{r}^{(\rom{1})}(n)=&\overline{h}_{r}^{(\rom{1})}(n-1)-(\alpha+\beta)\left(\overline{h}_{r}^{(\rom{1})}(n-2)-\overline{h}_{r}^{(\rom{1})}(n-1)\right)\\
&-(1-\alpha-\beta)\left(\overline{h}_{r}^{(\rom{1})}(n-3)-\overline{h}_{r}^{(\rom{1})}(n-2)\right)\\
&+\alpha^{k-1}(1-\alpha)(1-\beta)\left(\overline{h}_{r}^{(\rom{1})}(n-k-2)-\overline{h}_{r}^{(\rom{1})}(n-k-1)\right)\\
&-\alpha^{k}\left(\overline{h}_{r-1}^{(\rom{1})}(n-k-1)-\overline{h}_{r-1}^{(\rom{1})}(n-k)\right)\\
&+\alpha^{k-1}(1-\alpha-\alpha^{2}-\beta)\left(\overline{h}_{r-1}^{(\rom{1})}(n-k-2)-\overline{h}_{r-1}^{(\rom{1})}(n-k-1)\right)\\
&+\alpha^{k}(1-\alpha-\beta)\left(\overline{h}_{r-1}^{(\rom{1})}(n-k-3)-\overline{h}_{r-1}^{(\rom{1})}(n-k-2)\right).\\
\end{split}
\end{equation*}
\end{corollary}

Let us give some formulae expressing the generating functions of the first two moments of $T_{r,k}^{(\rom{1})}$ and $(T_{r,k}^{(\rom{1})})^{2}$ by means of the double generating function $H(z,w)$.

\begin{theorem}
The generating function of the means $E[T_{r,k}^{(\rom{1})}]$ is given by
\begin{equation}\label{eq:type1-first-mean-markov}
\begin{split}
\sum_{r=0}^{\infty}E[T_{r,k}^{(\rom{1})}]w^{r}=\frac{\left[\alpha(2-\alpha-\beta)+\alpha^{k}\{\alpha(\beta-q)-p\}\right]w+(1-\alpha)(p-\alpha)\alpha^{k}w^{2}}{(1-\alpha)(1-\beta)\alpha^{k}(w-1)^{2}}.
\end{split}
\end{equation}
\end{theorem}


We can establish an appropriate recurrent relation for the $\mu_{r}^{(\rom{1})}=E[T_{r,k}^{(\rom{1})}]$ of the random variable $T_{r,k}^{(\rom{1})}$ using the result established in (\ref{eq:type1-first-mean}), is given in the following result.

\begin{corollary}The $\mu_{r}^{(\rom{1})}=E[T_{r,k}^{(\rom{1})}]$, $r=1,\ 2, ...$ of the random variable $T_{r}^{(\rom{1})}$ satisfies the recurrence relation.
\begin{equation}\label{eq:recu-type1-mean-first}
\begin{split}
\mu_{r}^{(\rom{1})}=2\mu_{r-1}^{(\rom{1})}-\mu_{r-2}^{(\rom{1})},\ r\geq 3.
\end{split}
\end{equation}
with initial conditions $\mu_{0}^{(\rom{1})}=0$, $\mu_{1}^{(\rom{1})}=\frac{\alpha(2-\alpha-\beta)+\alpha^{k}\{\alpha(\beta-q)-p\}}{(1-\alpha)(1-\beta)\alpha^{k}}$ and\\
$\mu_{2}^{(\rom{1})}=\frac{2\alpha(2-\alpha-\beta)+\alpha^{k}\{2\alpha\beta-\alpha q-2\alpha+\alpha^{2}-p\}}{(1-\alpha)(1-\beta)\alpha^{k}}$.
\end{corollary}

\begin{proof}
It follows by equating the coefficients of $w^{r}$ on both sides of (\ref{eq:type1-first-mean-markov}).
\end{proof}

\begin{theorem}
The generating function of the means $E[(T_{r,k}^{(\rom{1})})^{2}]$ is given by
\begin{equation*}\label{eq: 1.1}
\begin{split}
\sum_{r=0}^{\infty}E[(T_{r,k}^{(\rom{1})})^{2}]w^{r}=\frac{a_{1}w+a_{2}\alpha^{k+1}w^{2}+a_{3}w^{3}}{\alpha^{2k}(1-\alpha)^{2}(1-\beta)^{2}(1-w)^{3}},
\end{split}
\end{equation*}
where
\begin{equation*}\label{eq: 1.1}
\begin{split}
a_{1}=&\alpha^{2k}(1-\alpha)(1-\beta)\{p+(\beta-q)\alpha\}+2\alpha^{2}(2-\alpha-\beta)^{2}\\
&-\alpha^{k+1}[2p(1-\alpha)(2-\alpha-\beta)+2k(1-\alpha)(1-\beta)(2-\alpha-\beta)\\
&+(1-\beta)\{\beta+\alpha(5-3\alpha-3\beta)\}],\\
a_{2}=&2k(1-\alpha)(1-\beta)(2-\alpha-\beta)+2p(1-\alpha)(2-\alpha-\beta)-(1+\alpha)\beta^{2}\\
&+(8-5\alpha)\alpha\beta+\beta-\alpha\{11-(13-4\alpha)\alpha\}\\
&-\alpha^{k-1}[2p(1-\alpha)\{1-(1-\alpha)\alpha-\beta\}\\
&+\alpha(1-\beta)\{2-\beta+\alpha(3-3\alpha-\beta)\}],\\
a_{3}=&\alpha^{2k}(1-\alpha)^{2}(p-\alpha)(1-2\alpha-\beta).\\
\end{split}
\end{equation*}
\end{theorem}


\begin{corollary}The $U_{r}^{(\rom{1})}=E[(T_{r,k}^{(\rom{1})})^{2}]$, $r=1,\ 2, ...$ of the random variable $(T_{r,k}^{(\rom{1})})^{2}$ satisfies the recurrence relation.
\begin{equation}\label{eq:recu-type1-mean-first}
\begin{split}
U_{r}^{(\rom{1})}=3U_{r-1}^{(\rom{1})}-3U_{r-2}^{(\rom{1})}+U_{r-3}^{(\rom{1})},\ r\geq 4.
\end{split}
\end{equation}
with initial conditions
\begin{equation*}\label{eq: 1.1}
\begin{split}
U_{1}^{(\rom{1})}=&\frac{c_{1}}{\alpha^{2k}(1-\alpha)^{2}(1-\beta)^{2}},\ U_{2}^{(\rom{1})}=\frac{c_{2}}{\alpha^{2k}(1-\alpha)^{2}(1-\beta)^{2}},\\ U_{3}^{(\rom{1})}=&\frac{c_{3}}{\alpha^{2k}(1-\alpha)^{2}(1-\beta)^{2}},
\end{split}
\end{equation*}
where
\begin{equation*}\label{eq: 1.1}
\begin{split}
c_{1}=&\alpha^{2k}(1-\alpha)(1-\beta)\{p+(\beta-q)\alpha\}+2\alpha^{2}(2-\alpha-\beta)^{2}\\
&-\alpha^{k+1}[2p(1-\alpha)(2-\alpha-\beta)+2k(1-\alpha)(1-\beta)(2-\alpha-\beta)\\
&+(1-\beta)\{\beta+\alpha(5-3\alpha-3\beta)\}],\\
c_{2}=&6\alpha^{4}-4\alpha^{k+4}+2\alpha^{k+3}\{11-2p-2k(1-\beta)-7\beta\}-12\alpha^{3}(2-\beta)\\
&6\alpha^{2}(2-\beta)^{2}+p\alpha^{2k}(1-\beta)-\alpha^{2k+3}(3-2p-3\beta)\\
&-\alpha^{2k+1}\{1-2p(2-\beta)-(3-2\beta)\beta\}+\alpha^{2k+2}(6-7p-10\beta+3p\beta+4\beta^{2})\\
&-2\alpha^{k+2}\{13-6k-6p-2(8-4k-p)\beta+(5-2k)\beta^{2}\}\\
&-2\alpha^{k+1}\{4p+2k(2-\beta)(1-\beta)+\beta-\beta(2p+\beta)\},\\
c_{3}=&12\alpha^{2}(2-\alpha-\beta)^{2}-\alpha^{k+1}[(3-\alpha)\alpha(7-4\alpha)+\beta-\alpha(24-11\alpha)\beta\\
&-(1-7\alpha)\beta^{2}+2p(1-\alpha)(2-\alpha-\beta)+2k(1-\alpha)(1-\beta)(2-\alpha-\beta)]\\
&+\alpha^{2k+1}[\alpha\{19-(7-\alpha)\alpha\}-1+4\beta-2\alpha\beta(13-5\alpha)-3(1-3\alpha)\beta^{2}]\\
&\alpha^{2k}[p(1-\alpha)\{1-\beta+\alpha(9-4\alpha-5\beta)\}]
\end{split}
\end{equation*}
\end{corollary}

\subsubsection{Exceed a threashold scheme}

%

\begin{proposition}
The double probability generating function $H^{(\rom{2})}(z,w)$ of $T_{r,k}^{(\rom{2})}$ is given by
\begin{equation}\label{eq:3.1}
\begin{split}
H^{(\rom{2})}(z,w)=\frac{P(z,w)}{Q(z,w)},
\end{split}
\end{equation}
where
\begin{equation*}\label{eq: 1.1}
\begin{split}
P(z,w)&=P(z)w+R(z),\ Q(z,w)=Q(z)w+R(z),\\
R(z)&=(1-\alpha)(1-\beta)\alpha^{k-1}z^{k+1}-(1-\alpha-\beta)z^{2}-(\alpha+\beta)z+1,\\
P(z)&=\alpha^{k}(\beta-q)z^{k+2}+(q\alpha-p-\alpha \beta)\alpha^{k-1}z^{k+1}+p\alpha^{k-1}z^{k},\\
Q(z)&=-(1-\alpha)(1-\beta)\alpha^{k-1}z^{k+1}.
\end{split}
\end{equation*}
\end{proposition}


In the following lemma we derive a recursive scheme for the evaluation of $H_{r}^{(\rom{2})}(z)$.
\begin{lemma}\label{lem:1}
The probability generating function $H_{r}^{(\rom{2})}(z)$ of the random variable $T_{r,k}^{(\rom{2})}$ is satisfies the recursive scheme
\begin{equation}\label{eq: 1.1}
\begin{split}
H_{r}^{(\rom{2})}(z)=\left[\frac{(1-\alpha)(1-\beta)\alpha^{k-1}z^{k+1}}{(1-\alpha)(1-\beta)\alpha^{k-1}z^{k+1}-(1-\alpha-\beta) z^{2}-(\alpha+\beta)z+1}\right] H_{r-1}^{(\rom{2})}(z)
\end{split}
\end{equation}
 with initial conditions $H_{0}^{(\rom{2})}(z)=1$, $H_{1}^{(\rom{2})}(z)=\frac{(\beta-q)\alpha^{k+1}z^{k+2}+\alpha^{k}(q-p\alpha-\beta)z^{k+1}+\alpha^{k}pz^{k}}{(1-\alpha)(1-\beta)\alpha^{k-1}z^{k+1}-(1-\alpha-\beta)z^{2}-(\alpha+\beta)z+1.}$
\end{lemma}

\begin{proof}
It follows by equating the coefficients of $w^{r}$ on both sides of (\ref{eq:3.1}).
\end{proof}

%
%
%
%
%

An efficient recursive scheme for the evaluation of the probability mass function of $T_{r,k}^{(\rom{2})}$, ensuing from the result established in \ref{lem:1}, is given in the following theorem.

\begin{theorem}
The probability mass function $h_{r}^{(\rom{2})}(n)$ of the random variable $T_{r,k}^{(\rom{2})}$ satisfies the recursive scheme
\begin{equation}\label{eq: 1.1}
\begin{split}
h_{r}^{(\rom{2})}(n)=& (\alpha+\beta)h_{r}^{(\rom{2})}(n-1)+(1-\alpha-\beta)\alpha h_{r}^{(\rom{2})}(n-2)-\alpha^{k-1}(1-\alpha)(1-\beta)\\
&\times h_{r}^{(\rom{2})}(n-k-1)+\alpha^{k-1}(1-\alpha)(1-\beta)h_{r-1}^{(\rom{2})}(n-k-1).
\end{split}
\end{equation}
with initial conditions $h_{0}^{(\rom{2})}(n)=\delta_{n,0}$ and
\begin{equation*}
\begin{split}
h_{1}^{(\rom{2})}(n)=\left\{
  \begin{array}{ll}
    0  & \text{if $0\leq n<k$,} \\
    p\alpha^{k} & \text{if $n=k$,} \\
    q\alpha^{k}(1-\beta) & \text{if $n=k+1$,} \\
    \alpha^{k}(1-\beta)\left\{\beta q+p(1-\alpha)\right\} & \text{if $n=k+2$.}\\
  \end{array}
\right.
\end{split}
\end{equation*}

\end{theorem}

\begin{proof}
It suffices to replace $H_{r}(z)$,in the recursive formulae given in Lemma \ref{lem:1} by the power series
\begin{equation}\label{eq: 3.1}
\begin{split}
H_{r}^{(\rom{2})}(z)=\sum_{n=0}^{\infty}P(T_{r,k}^{(\rom{2})}=n)z^{n}=\sum_{n=0}^{\infty}h_{r}^{(\rom{2})}(n)z^{n},
\end{split}
\end{equation}
and then equating the coefficients of $z^{n}$ on both sides of the resulting identities.
\end{proof}

\begin{corollary}The tail probability function $\overline{h}_{r}^{(\rom{2})}(n)=P(T_{r,k}^{(\rom{2})}>n)$ of $T_{r,k}^{(\rom{2})}$ satisfies the recurrence relation.
\begin{equation*}\label{eq: 1.1}
\begin{split}
\overline{h}_{r}^{(\rom{2})}(n)=&(1+\alpha+\beta)\overline{h}_{r}^{(\rom{2})}(n-1)-(\alpha+\beta)\overline{h}_{r}^{(\rom{2})}(n-2)\\
&-\alpha(1-\alpha-\beta)\left\{\overline{h}_{r}^{(\rom{2})}(n-3)-\overline{h}_{r}^{(\rom{2})}(n-2)\right\}\\
&+\alpha^{k-1}(1-\alpha)(1-\beta)\left\{\overline{h}_{r}^{(\rom{2})}(n-k-2)-\overline{h}_{r}^{(\rom{2})}(n-k-1)\right\}\\
&-\alpha^{k-1}(1-\alpha)(1-\beta)\left\{\overline{h}_{r-1}^{(\rom{2})}(n-k-2)-\overline{h}_{r-1}^{(\rom{2})}(n-k-1)\right\}.\\
\end{split}
\end{equation*}
\end{corollary}

Let us give some formulae expressing the generating functions of the first two moments of $T_{r,k}^{(\rom{2})}$ and $(T_{r,k}^{(\rom{2})})^{2}$ by means of the double generating function $H(z,w)$.

\begin{theorem}
The generating function of the means $E[T_{r,k}^{(\rom{2})}]$ is given by
\begin{equation}\label{eq: 1.1}
\begin{split}
\sum_{r=0}^{\infty}E[T_{r,k}^{(\rom{2})}]w^{r}=\dfrac{\splitfrac{\{\alpha^{k-1}(\alpha\beta-q\alpha-p)+(2-\alpha-\beta)\}w}{+\alpha^{k-1}(1-\alpha)(p-\alpha)w^{2}}}{(1-\alpha)(1-\beta)\alpha^{k-1}(1-w)^{2}}
\end{split}
\end{equation}
\end{theorem}

\begin{corollary}The $\mu_{r}^{(\rom{2})}=E[T_{r,k}^{(\rom{2})}]$, $r=1,\ 2, ...$ of the random variable $T_{r}^{(\rom{2})}$ satisfies the recurrence relation.
\begin{equation}\label{eq:recu-type1-mean-first}
\begin{split}
\mu_{r}^{(\rom{2})}=2\mu_{r-1}^{(\rom{2})}-\mu_{r-2}^{(\rom{2})},\ r\geq 3.
\end{split}
\end{equation}
with initial conditions $\mu_{0}^{(\rom{2})}=0$, $\mu_{1}^{(\rom{2})}=\frac{\alpha^{k-1}(\alpha\beta-q\alpha-p)+(2-\alpha-\beta)}{(1-\alpha)(1-\beta)\alpha^{k-1}}$ and\\
$\mu_{2}^{(\rom{2})}=\frac{\alpha^{k-1}(2\alpha\beta-3q\alpha-2p-p\alpha+\alpha^{2})+2(2-\alpha-\beta)}{(1-\alpha)(1-\beta)\alpha^{k-1}}$.
\end{corollary}


\begin{theorem}
The generating function of the means $E[(T_{r,k}^{(\rom{2})})^{2}]$ is given by
\begin{equation*}\label{eq: 1.1}
\begin{split}
\sum_{r=0}^{\infty}E[(T_{r,k}^{(\rom{2})})^{2}]w^{r}=\frac{A_{1}w+A_{2}w^{2}+a_{1}a_{2}a_{3}\alpha^{2k}w^{3}}{\alpha^{2k}a_{1}^{2}a_{2}^{2}(1-w)^{3}},
\end{split}
\end{equation*}
where
\begin{equation*}\label{eq: 1.1}
\begin{split}
a_{1}=&1-\alpha,\ a_{2}=1-\beta,\ a_{3}=(\beta-q)\alpha+p,\ a_{4}=\beta+\alpha(5-3\alpha-3\beta),\\
A_{1}=&a_{1}a_{2}a_{3}\alpha^{2k}+2\alpha(a_{1}+a_{2})^{2}-\alpha^{k+1}\{2pa_{1}(a_{1}+a_{2})+2ka_{1}a_{2}(a_{1}+a_{2})-a_{2}a_{4}\},\\
A_{2}=&\alpha^{k}\left[2k\alpha a_{1}a_{2}(a_{1}+a_{2})-2\alpha^{k}a_{1}a_{2}a_{3}+\alpha\{2pa_{1}(a_{1}+a_{2})+a_{2}a_{4}\}\right].\\
\end{split}
\end{equation*}
\end{theorem}


\begin{corollary}The $U_{r}^{(\rom{2})}=E[(T_{r,k}^{(\rom{2})})^{2}]$, $r=1,\ 2, ...$ of the random variable $(T_{r,k}^{(\rom{2})})^{2}$ satisfies the recurrence relation.
\begin{equation}\label{eq:recu-type1-mean-first}
\begin{split}
U_{r}^{(\rom{2})}=3U_{r-1}^{(\rom{2})}-3U_{r-2}^{(\rom{2})}+U_{r-3}^{(\rom{2})},\ r\geq 4.
\end{split}
\end{equation}
with initial conditions
\begin{equation*}\label{eq: 1.1}
\begin{split}
U_{1}^{(\rom{2})}=&\frac{B_{1}}{\alpha^{2k}(1-\alpha)^{2}(1-\beta)^{2}},\ U_{2}^{(\rom{2})}=\frac{B_{2}}{\alpha^{2k}(1-\alpha)^{2}(1-\beta)^{2}},\\ U_{3}^{(\rom{2})}=&\frac{B_{3}}{\alpha^{2k}(1-\alpha)^{2}(1-\beta)^{2}},
\end{split}
\end{equation*}
where
\begin{equation*}\label{eq: 1.1}
\begin{split}
a_{1}=&1-\alpha,\ a_{2}=1-\beta,\ a_{3}=(\beta-q)\alpha+p,\ a_{4}=\beta+\alpha(5-3\alpha-3\beta),\\
B_{1}=&a_{1}a_{2}a_{3}\alpha^{2k}+2\alpha(a_{1}+a_{2})^{2}-\alpha^{k+1}\{2pa_{1}(a_{1}+a_{2})+2ka_{1}a_{2}(a_{1}+a_{2})-a_{2}a_{4}\},\\
B_{2}=&a_{1}a_{2}a_{3}\alpha^{2k}+6\alpha^{2}(a_{1}+a_{2})^{2}-2\alpha^{k+1}\{2pa_{1}(a_{1}+a_{2})+2ka_{1}a_{2}(a_{1}+a_{2})-a_{2}a_{4}\},\\
B_{3}=&a_{1}a_{2}a_{3}\alpha^{2k}+12\alpha^{2}(a_{1}+a_{2})^{2}-3\alpha^{k+1}\{2pa_{1}(a_{1}+a_{2})+2ka_{1}a_{2}(a_{1}+a_{2})+a_{2}a_{4}\}.\\
\end{split}
\end{equation*}
\end{corollary}

\subsubsection{Overlapping scheme}

%

\begin{proposition}
The double probability generating function $H^{(\rom{3})}(z,w)$ of $T_{r,k}^{(\rom{3})}$ is given by
\begin{equation}\label{eq:3.1}
\begin{split}
H^{(\rom{3})}(z,w)=\frac{P(z,w)}{Q(z,w)},
\end{split}
\end{equation}
where
\begin{equation*}\label{eq: 1.1}
\begin{split}
P(z,w)&=P(z)w+R(z),\ Q(z,w)=Q(z)w+R(z),\\
R(z)&=(1-\alpha)(1-\beta)\alpha^{k-1}z^{k+1}-(1-\alpha-\beta)z^{2}-(\alpha+\beta)z+1,\\
P(z)&=\alpha^{k}(\beta-q)z^{k+2}+\alpha^{k-1}\{p+(\beta-q)\alpha\}z^{k+1}+p\alpha^{k-1}z^{k}\\
&+\alpha(1-\alpha-\beta)z^{3}+\alpha(\alpha+\beta)z^{2}-\alpha z,\\
Q(z)&=-(1-\alpha)(1-\beta)\alpha^{k-1}z^{k+1}+\alpha(1-\alpha-\beta)z^{3}+\alpha(\alpha+\beta)z^{2}-\alpha z.
\end{split}
\end{equation*}
\end{proposition}


In the following lemma we derive a recursive scheme for the evaluation of $H_{r}(z)$.
\begin{lemma}\label{lem:1}
The probability generating function $H_{r}^{(\rom{3})}(z)$ of the random variable $T_{r,k}^{(\rom{3})}$ is satisfies the recursive scheme
\begin{equation}\label{eq: 1.1}
\begin{split}
H_{r}^{(\rom{3})}(z)=\left[\frac{(1-\alpha)(1-\beta)\alpha^{k-1}z^{k+1}-\alpha(1-\alpha-\beta)z^{3}-\alpha(\alpha+\beta)z^{2}+\alpha z}{(1-\alpha)(1-\beta)\alpha^{k-1}z^{k+1}-(1-\alpha-\beta)z^{2}-(\alpha+\beta)z+1}\right] H_{r-1}^{(\rom{3})}(z)
\end{split}
\end{equation}
with initial conditions $H_{0}^{(\rom{3})}(z)=1$, $H_{1}^{(\rom{3})}(z)=\frac{(\beta-q)\alpha^{k+1}z^{k+2}-\alpha^{k}(p\alpha+\beta-q)z^{k+1}+p\alpha^{k}z^{k}}{(1-\alpha)(1-\beta)\alpha^{k-1}z^{k+1}-(1-\alpha-\beta)z^{2}-(\alpha+\beta)z+1.}$
\end{lemma}

\begin{proof}
It follows by equating the coefficients of $w^{r}$ on both sides of (\ref{eq:3.1}).
\end{proof}

%
%
%
%
%

An efficient recursive scheme for the evaluation of the probability mass function of $T_{r,k}^{(\rom{3})}$, ensuing from the result established in \ref{lem:1}, is given in the following theorem.

\begin{theorem}
The probability mass function $h_{r}^{(\rom{3})}(n)$ of the random variable $T_{r,k}^{(\rom{3})}$ satisfies the recursive scheme
\begin{equation}\label{eq: 1.1}
\begin{split}
h_{r}^{(\rom{3})}(n)=& (\alpha+\beta)h_{r}^{(\rom{3})}(n-1)+(1-\alpha-\beta)h_{r}^{(\rom{3})}(n-2)\\
&-\alpha^{k-1}(1-\alpha)(1-\beta)h_{r}^{(\rom{3})}(n-k-1)+ \alpha h_{r-1}^{(\rom{3})}(n-1)\\
&-\alpha(\alpha+\beta)h_{r-1}^{(\rom{3})}(n-2)-\alpha(1-\alpha-\beta)h_{r-1}^{(\rom{3})}(n-3)\\
&+\alpha^{k-1}(1-\alpha)(1-\beta)h_{r-1}^{(\rom{3})}(n-k-1),\ n>k+1
\end{split}
\end{equation}
with initial conditions $h_{0}^{(\rom{3})}(n)=\delta_{n,0}$ and
\begin{equation*}
\begin{split}
h_{1}^{(\rom{3})}(n)=\left\{
  \begin{array}{ll}
    0  & \text{if $0\leq n<k$,} \\
    p\alpha^{k} & \text{if $n=k$,} \\
    q\alpha^{k}(1-\beta) & \text{if $n=k+1$,} \\
    \alpha^{k}(1-\beta)\left\{\beta q+p(1-\alpha)\right\} & \text{if $n=k+2$.}\\
  \end{array}
\right.
\end{split}
\end{equation*}
\end{theorem}

\begin{proof}
It suffices to replace $H_{r}(z)$,in the recursive formulae given in Lemma \ref{lem:1} by the power series
\begin{equation}\label{eq: 3.1}
\begin{split}
H_{r}(z)=\sum_{n=0}^{\infty}P(T_{r,k}^{(\rom{3})}=n)z^{n}=\sum_{n=0}^{\infty}h_{r}(n)z^{n},
\end{split}
\end{equation}
and then equating the coefficients of $z^{n}$ on both sides of the resulting identities.
\end{proof}

\begin{corollary}The tail probability function $\overline{h}_{r}^{(\rom{3})}(n)=P(T_{r,k}^{(\rom{3})}>n)$ of $T_{r,k}^{(\rom{3})}$ satisfies the recurrence relation.
\begin{equation*}\label{eq: 1.1}
\begin{split}
\overline{h}_{r}^{(\rom{3})}(n)=&(1+\alpha+\beta)\overline{h}_{r}^{(\rom{3})}(n-1)-(1-\alpha-\beta)\overline{h}_{r}^{(\rom{3})}(n-3)\\
&-\alpha^{k-1}(1-\alpha)(1-\beta)\left\{\overline{h}_{r}^{(\rom{3})}(n-k-2)-\overline{h}_{r}^{(\rom{3})}(n-k-1)\right\}\\
&+\alpha\overline{h}_{r-1}^{(\rom{3})}(n-1)-\alpha(1+\alpha+\beta)\overline{h}_{r-1}^{(\rom{3})}(n-2)+\alpha\overline{h}_{r-1}^{(\rom{3})}(n-3)\\
&+\alpha(1-\alpha-\beta)\overline{h}_{r-1}^{(\rom{3})}(n-4)\\
&-\alpha^{k-1}(1-\alpha)(1-\beta)\left\{\overline{h}_{r-1}^{(\rom{3})}(n-k-2)-\overline{h}_{r-1}^{(\rom{3})}(n-k-1)\right\}.\\
\end{split}
\end{equation*}
\end{corollary}

Let us give some formulae expressing the generating functions of the first two moments of $T_{r,k}^{(\rom{3})}$ and $(T_{r,k}^{(\rom{3})})^{2}$ by means of the double generating function $H(z,w)$.

\begin{theorem}
The generating function of the means $E[T_{r,k}^{(\rom{3})}]$ is given by
\begin{equation}\label{eq: 1.1}
\begin{split}
\sum_{r=0}^{\infty}E[T_{r,k}^{(\rom{3})}]w^{r}=\frac{a_{1}w^{2}+a_{2}w}{\alpha^{k-1}(1-\alpha)(1-\beta)(1-w)^{2}}.
\end{split}
\end{equation}
where
\begin{equation*}\label{eq: 1.1}
\begin{split}
a_{1}=&\alpha^{k-1}\{(\beta-q)\alpha-p\}+(2-\alpha-\beta),\\
a_{2}=&(q-\beta)\alpha^{k}+p\alpha^{k-1}+\alpha^{2}-(2-\beta)\alpha.\\
\end{split}
\end{equation*}
\end{theorem}


\begin{corollary}The $\mu_{r}^{(\rom{3})}=E[T_{r,k}^{(\rom{3})}]$, $r=1,\ 2, ...$ of the random variable $T_{r}^{(\rom{3})}$ satisfies the recurrence relation.
\begin{equation}\label{eq:recu-type1-mean-first}
\begin{split}
\mu_{r}^{(\rom{3})}=2\mu_{r-1}^{(\rom{3})}-\mu_{r-2}^{(\rom{3})},\ r\geq 3.
\end{split}
\end{equation}
with initial conditions $\mu_{0}^{(\rom{3})}=0$, $\mu_{1}^{(\rom{3})}=\frac{(q-\beta)\alpha^{k}+p\alpha^{k-1}+\alpha^{2}-(2-\beta)\alpha}{(1-\alpha)(1-\beta)\alpha^{k-1}}$ and\\
$\mu_{2}^{(\rom{3})}=\frac{\alpha^{k-1}\{(\beta-q)\alpha-p)+(2-\alpha)(2-\alpha-\beta)}{(1-\alpha)(1-\beta)\alpha^{k-2}}$.
\end{corollary}

\begin{theorem}
The generating function of the means $E[(T_{r,k}^{(\rom{3})})^{2}]$ is given by
\begin{equation*}\label{eq: 1.1}
\begin{split}
\sum_{r=0}^{\infty}E[(T_{r,k}^{(\rom{3})})^{2}]w^{r}=\frac{a_{1}w+a_{2}w^{2}+a_{3}w^{3}}{\alpha^{2k}(1-\alpha)^{2}(1-\beta)^{2}(1-w)^{3}},
\end{split}
\end{equation*}
where
\begin{equation*}\label{eq: 1.1}
\begin{split}
a_{1}=&\alpha^{2k}(1-\alpha)(1-\beta)\{p+(\beta-q)\alpha\}+2\alpha^{2}(2-\alpha-\beta)^{2}\\
&-\alpha^{k+1}\Big[2p(1-\alpha)(2-\alpha-\beta)+2k(1-\alpha)(1-\beta)(2-\alpha-\beta)\\
&-(1-\beta)\{\beta+\alpha(5-3\alpha-3\beta)\}\Big],\\
a_{2}=&-2\alpha^{2k}(1-\alpha)(1-\beta)\{p+(\beta-q)\alpha\}+\{(1-\beta)(2k-5)+2p\}\alpha^{k+4}\\
&+\alpha^{k+3}\Big[(1-\beta)\{5(1-\beta)-4p-2k(2-\beta)+3\}+2p\beta\Big]\\
&-\{2p+(2k-1)(1-\beta)\}\alpha^{k+2}\\
&+\{2p(2-\beta)+2k(2-\beta)(1-\beta)+\beta(1-\beta)\}\alpha^{k+1}\\
&-4\alpha^{5}-4\alpha^{3}(2-\beta)(2-2\alpha-\beta),\\
a_{3}=&\alpha^{2k}(1-\alpha)(1-\beta)\{p+(\beta-q)\alpha\}+2\alpha^{4}(2-\alpha-\beta)^{2}\\
&-\alpha^{k+2}\Big[2p(1-\alpha)(2-\alpha-\beta)+2k(1-\alpha)(1-\beta)(2-\alpha-\beta)\\
&-(1-\beta)\{4-3\beta-\alpha(11-5\alpha-5\beta)\}\Big].
\end{split}
\end{equation*}
\end{theorem}


\begin{corollary}The $U_{r}^{(\rom{3})}=E[(T_{r,k}^{(\rom{3})})^{2}]$, $r=1,\ 2, ...$ of the random variable $(T_{r,k}^{(\rom{3})})^{2}$ satisfies the recurrence relation.
\begin{equation}\label{eq:recu-type1-mean-first}
\begin{split}
U_{r}^{(\rom{3})}=3U_{r-1}^{(\rom{3})}-3U_{r-2}^{(\rom{3})}+U_{r-3}^{(\rom{3})},\ r\geq 4.
\end{split}
\end{equation}
with initial conditions
\begin{equation*}\label{eq: 1.1}
\begin{split}
U_{1}^{(\rom{3})}=&\frac{B_{1}}{\alpha^{2k}(1-\alpha)^{2}(1-\beta)^{2}},\ U_{2}^{(\rom{3})}=\frac{B_{2}}{\alpha^{2k}(1-\alpha)^{2}(1-\beta)^{2}},\\ U_{3}^{(\rom{3})}=&\frac{B_{3}}{\alpha^{2k}(1-\alpha)^{2}(1-\beta)^{2}},
\end{split}
\end{equation*}
where
\begin{equation*}\label{eq: 1.1}
\begin{split}
B_{1}=&\alpha^{2k}(1-\alpha)(1-\beta)\{p+(\beta-q)\alpha\}+2\alpha^{2}(2-\alpha-\beta)^{2}\\
&-\alpha^{k+1}\Big[2p(1-\alpha)(2-\alpha-\beta)+2k(1-\alpha)(1-\beta)(2-\alpha-\beta)\\
&-(1-\beta)\{\beta+\alpha(5-3\alpha-3\beta)\}\Big],\\
B_{2}=&\alpha^{2k}(1-\alpha)(1-\beta)\{p+(\beta-q)\alpha\}+2\alpha^{2}(3-2\alpha)(2-\alpha-\beta)^{2}\\
&-\alpha^{k+1}\Big[2p(2-\alpha)(1-\alpha)(2-\alpha-\beta)+2k(2-\alpha)(1-\alpha)(1-\beta)(2-\alpha-\beta)\\
&-(1-\beta)\{\alpha(2-\alpha)(7-5\alpha)+(5\alpha^{2}-9\alpha+2)\beta\},\\
B_{3}=&\alpha^{2k}(1-\alpha)(1-\beta)\{p+(\beta-q)\alpha\}+2\alpha^{2}(\alpha^{2}-6\alpha+6)(2-\alpha-\beta)^{2}\\
&-\alpha^{k+1}\Big[2p(1-\alpha)(3-2\alpha)(2-\alpha-\beta)\\
&+2k(1-\alpha)(2-2\alpha)(1-\beta)(2-\alpha-\beta)\\
&-(1-\beta)(10\alpha^{3}+10\alpha^{2}\beta-31\alpha^{2}-15\alpha\beta+23\alpha+3\beta)\Big].\\
\end{split}
\end{equation*}
\end{corollary}

\section{Distribution of the Number of success runs}
In this section, we study the distributions of number of runs. Let $N_{n}$ be a random variable denoting the number of occurrences of runs in the sequence of $n$ trials. Its probability mass function will be defined by
\begin{equation}\label{eq: 1.1}
\begin{split}
g_{n}^{(a)}(x)&= \mathbf{P}(X_{n}^{(a)}=x),\ n\geq 0,\ a=\rom{1},\rom{2},\rom{3},
\end{split}
\end{equation}
and its single and double probability generating functions of $X_{n}$, $(n=0,\ 1,\ 2,\ ...,\ a=\rom{1},\rom{2},\rom{3})$ defined by
\begin{equation}\label{eq: 1.1}
\begin{split}
G_{n}^{(a)}(w)&=\sum_{x=0}^{\infty}g_{n}^{(a)}(x)w^{x},\ x\geq 0,\\
G^{(a)}(z,w)&=\sum_{n=0}^{\infty} G_{n}^{(a)}(w)z^{n}.
\end{split}
\end{equation}

 The random variable $N_{n}$ is closely related to the random variable $T_{r,k}$ (see Feller 1968). Koutras (1997) and Chang \textit{et} \textit{al}. (2012) established the dual relationship, the double generating function of $N_{n}^{(a)}$ which can be expressed in terms of the p.g.f. of the waiting time $T_{r,k}^{(a)}$, $G(z,w)=\frac{(w-1)H(z,w)+1}{w(1-z)}$ and $G(z,w)=\frac{1}{1-z}\left[1-H(z)\frac{1-w}{1-wA(z)}\right].$ Using this result, we derive recursive schemes for the p.g.f. and p.m.f. of $N_{n}^{(a)}$, $a=\rom{1},\rom{2},\rom{3}$, which is coincident with $N_{n,k}$, $G_{n,k}$  and $M_{n,k}$, respectively.

\subsection{I.i.d. trials}
\subsubsection{Non-overlapping scheme}

\begin{proposition}
The double probability generating function $G(z,w)$ of $N_{n}^{(\rom{1})}$ is given by
\begin{equation}\label{eq:3.1}
\begin{split}
G^{(\rom{1})}(z,w)=\frac{1-p^{k} z^{k}}{1-z-p^{k}wz^{k}+(q+pw)p^{k}z^{k+1}}.
\end{split}
\end{equation}
\end{proposition}


In the following lemma we derive a recursive scheme for the evaluation of $G_{n}(w)$.
\begin{lemma}\label{lem:1}
The probability generating function $G_{n}(w)$ of the random variable $N_{n}^{(\rom{1})}$ is satisfies the recursive scheme
\begin{equation}\label{eq: 1.1}
\begin{split}
G_{n}^{(\rom{1})}(w)=G_{n-1}^{(\rom{1})}(w)+p^{k}w G_{n-k}^{(\rom{1})}(w)-(q+pw)p^{k}G_{n-k-1}^{(\rom{1})}(w),\ \text{for}\ n>k,
\end{split}
\end{equation}
with initial conditions
\begin{equation*}
\begin{split}
G_{n}^{(\rom{1})}(w)=\left\{
  \begin{array}{ll}
    1  & \text{if $0\leq n<k$,} \\
    1-p^{k}+p^{k}w & \text{if $n=k$.} \\
  \end{array}
\right.
\end{split}
\end{equation*}
\end{lemma}

\begin{proof}
It follows by equating the coefficients of $z^{n}$ on both sides of (\ref{eq:3.1}).
\end{proof}

%
%
%
%
%

An efficient recursive scheme for the evaluation of the probability mass function of $N_{n}^{(\rom{1})}$, ensuing from the result established in \ref{lem:1}, is given in the following theorem.

\begin{theorem}
The probability mass function $g_{n}(x)$ of the random variable $N_{n}^{(\rom{1})}$ satisfies the recursive scheme
\begin{equation}\label{eq: 1.1}
\begin{split}
g_{n}^{(\rom{1})}(x)= g_{n-1}^{(\rom{1})}(x)+p^{k}g_{n-k}^{(\rom{1})}(x-1)-p^{k}q g_{n-k-1}^{(\rom{1})}(x)-p^{k+1}g_{n-k-1}^{(\rom{1})}(x-1),\ n>k,
\end{split}
\end{equation}
with initial conditions
\begin{equation}\label{eq: 1.1}
\begin{split}
g_{n}^{(\rom{1})}(x)&=\delta_{x,0},\ \text{for}\ 0\leq n<k,\\
g_{k}^{(\rom{1})}(0)&=1-p^{k},\ g_{k}^{(\rom{1})}(1)=p^{k},\ g_{k}^{(\rom{1})}(x)=0\ \text{for}\ x\geq 2.
\end{split}
\end{equation}
\end{theorem}

\begin{proof}
It suffices to replace $G_{n}(w)$, in the recursive formulae given in Lemma \ref{lem:1} by the power series
\begin{equation}\label{eq: 3.1}
\begin{split}
G_{n}(w)=\sum_{x=0}^{\infty}P(X_{n}=x)w^{x}=\sum_{x=0}^{\infty}g_{n}(x)w^{x},
\end{split}
\end{equation}
and then equating the coefficients of $w^{x}$ on both sides of the resulting identities.
\end{proof}

Let us give some formulae expressing the generating functions of the first two moments of $N_{n}^{(\rom{1})}$ and $(N_{n}^{(\rom{1})})^{2}$ by means of the double generating function $H(z,w)$ and $G(z,w)$.

\begin{theorem}
The generating function of the means $E[N_{n}^{(\rom{1})}]$ is given by
\begin{equation}\label{eq: 1.1}
\begin{split}
\sum_{n=0}^{\infty}E[N_{n}^{(\rom{1})}]z^{n}=\frac{(pz)^{k}(1-pz)}{(1-z)^{2}(1-p^{k}z^{k}).}
\end{split}
\end{equation}
\end{theorem}


\begin{corollary}The $\tau_{n}^{(\rom{1})}=E[N_{n}^{(\rom{1})}]$, $n=1,\ 2, ...$ of the random variable $N_{n}^{(\rom{1})}$ satisfies the recurrence relation.
\begin{equation}\label{eq:recu-type1-mean-first}
\begin{split}
\tau_{n}^{(\rom{1})}=2\tau_{n-1}^{(\rom{1})}-\tau_{n-2}^{(\rom{1})}+p^{k}\tau_{n-k}^{(\rom{1})}-2p^{k}\tau_{n-k-1}^{(\rom{1})}+p^{k}\tau_{n-k-2}^{(\rom{1})},\  \text{for}\ n>k+1,
\end{split}
\end{equation}
with initial conditions $\tau_{n}^{(\rom{1})}=0$, for $0\leq n \leq k-1$, $\tau_{k}^{(\rom{1})}=p^{k}$ and $\tau_{k+1}^{(\rom{1})}=p^{k}(2-p)$.
\end{corollary}

\begin{theorem}
The generating function of the means $E[(N_{n}^{(\rom{1})})^{2}]$ is given by
\begin{equation*}\label{eq: 1.1}
\begin{split}
\sum_{n=0}^{\infty}E[(N_{n}^{(\rom{1})})^{2}]z^{n}=\frac{(pz)^{k}(1-pz)\{1-z+p^{k}z^{k}+(q-p)p^{k}z^{k+1}\}}{(1-z)^{3}(1-p^{k}z^{k})^{2}},
\end{split}
\end{equation*}
\end{theorem}


\begin{corollary}The $\nu_{n}^{(\rom{1})}=E[(N_{n}^{(\rom{1})})^{2}]$, $n=1,\ 2, ...$ of the random variable $(N_{n}^{(\rom{1})})^{2}$ satisfies the recurrence relation.
\begin{equation}\label{2.1}
\begin{split}
\nu_{n}^{(\rom{1})}=&3\nu_{n-1}^{(\rom{1})}-3\nu_{n-2}^{(\rom{1})}+\nu_{n-3}^{(\rom{1})}+2p^{k}\nu_{n-k}^{(\rom{1})}-6p^{k}\nu_{n-k-1}^{(\rom{1})}+6p^{k}\nu_{n-k-2}^{(\rom{1})}\\
&-2p^{k}\nu_{n-k-3}^{(\rom{1})}-p^{2k}\nu_{n-2k}^{(\rom{1})}+3p^{2k}\nu_{n-2k-1}^{(\rom{1})}-3p^{2k}\nu_{n-2k-2}^{(\rom{1})}\\
&+p^{2k}\nu_{n-2k-3}^{(\rom{1})},\  \text{for}\ n>2k+3,
\end{split}
\end{equation}
with initial conditions
\begin{equation*}
\begin{split}
\nu_{n}^{(\rom{1})}=\left\{
  \begin{array}{ll}
    0  & \text{if $0\leq n<k$,} \\
    \frac{n^{2}+(3-2k)n+(k-1)(k-p2)}{2}p^{k}  & \text{if $k\leq n<2k$,} \\
    3p^{k}(k^{2}-5k+12)+3p^{2k}  & \text{if $n=2k$,} \\
    \frac{3k(k+5)}{2}+6p^{k}-3p^{2k}  & \text{if $n=2k+1$,} \\
    \frac{k(k+7)}{2}+6+3p^{2k}  & \text{if  $n=2k+2$,} \\
    \frac{k(k+9)}{2}+10-p^{2k}& \text{if $n=2k+3$.} \\
  \end{array}
\right.
\end{split}
\end{equation*}

\end{corollary}

\subsubsection{Exceed a threashold scheme}

\begin{proposition}
The double probability generating function $G(z,w)$ of $N_{n}^{(\rom{2})}$ is given by
\begin{equation}\label{eq:3.1}
\begin{split}
G^{(\rom{2})}(z,w)=\frac{1-p^{k}(1-w)z^{k}}{1-z+p^{k}q(1-w)z^{k+1}}.
\end{split}
\end{equation}

\end{proposition}


In the following lemma we derive a recursive scheme for the evaluation of $G_{n}(w)$.
\begin{lemma}\label{lem:1}
The probability generating function $G_{n}(w)$ of the random variable $N_{n}^{(\rom{2})}$ is satisfies the recursive scheme
\begin{equation}\label{eq: 1.1}
\begin{split}
G_{n}^{(\rom{2})}(w)=G_{n-1}^{(\rom{2})}(w)-p^{k}q(1-w) G_{n-k-1}^{(\rom{2})}(w),\ \text{for}\ n>k
\end{split}
\end{equation}
with initial conditions
\begin{equation*}
\begin{split}
G_{n}^{(\rom{2})}(w)=\left\{
  \begin{array}{ll}
    1  & \text{if $0\leq n<k$,} \\
   1-p^{k}+p^{k}w & \text{if $n=k$.} \\
  \end{array}
\right.
\end{split}
\end{equation*}
\end{lemma}

\begin{proof}
It follows by equating the coefficients of $z^{n}$ on both sides of (\ref{eq:3.1}).
\end{proof}

%
%
%
%
%

An efficient recursive scheme for the evaluation of the probability mass function of $N_{n}^{(\rom{2})}$, ensuing from the result established in \ref{lem:1}, is given in the following theorem.

\begin{theorem}
The probability mass function $g_{n}^{(\rom{2})}(x)$ of the random variable $N_{n}^{(\rom{2})}$ satisfies the recursive scheme
\begin{equation}\label{eq: 1.1}
\begin{split}
g_{n}^{(\rom{2})}(x)=&g_{n-1}^{(\rom{2})}(x)-p^{k}q g_{n-k-1}^{(\rom{2})}(x)+p^{k}qg_{n-k-1}^{(\rom{2})}(x-1),\ \text{for}\ n>k,
\end{split}
\end{equation}
with initial conditions
\begin{equation*}
\begin{split}
g_{n}^{(\rom{2})}(x)=&\delta_{0,x},\ \ \text{for}\ \ 0\leq n<k,\\
g_{k}^{(\rom{2})}(0)=&1-p^{k},\ \ g_{k}^{(\rom{2})}(1)=p^{k},\ \ g_{k}^{(\rom{2})}(x)=0,\ \ \text{for}\ \ x>1.
\end{split}
\end{equation*}
\end{theorem}

\begin{proof}
It suffices to replace $G_{n}(w)$, in the recursive formulae given in Lemma \ref{lem:1} by the power series
\begin{equation}\label{eq: 3.1}
\begin{split}
G_{n}^{(\rom{2})}(w)=\sum_{x=0}^{\infty}P(X_{n}^{(\rom{2})}=x)w^{x}=\sum_{x=0}^{\infty}g_{n}^{(\rom{2})}(x)w^{x},
\end{split}
\end{equation}
and then equating the coefficients of $w^{x}$ on both sides of the resulting identities.
\end{proof}

Let us give some formulae expressing the generating functions of the first two moments of $N_{n}^{(\rom{2})}$ and $(N_{n}^{(\rom{2})})^{2}$ by means of the double generating function $H(z,w)$ and $G(z,w)$.

\begin{theorem}
The generating function of the means $E[N_{n}^{(\rom{2})}]$ is given by
\begin{equation}\label{eq: 1.1}
\begin{split}
\sum_{n=0}^{\infty}E[N_{n}^{(\rom{2})}]z^{n}=\frac{p^{k}z^{k}(1-pz)}{(1-z)^{2}}.
\end{split}
\end{equation}
\end{theorem}


\begin{corollary}The $\tau_{n}^{(\rom{2})}=E[N_{n}^{(\rom{2})}]$, $n=1,\ 2, ...$ of the random variable $N_{n}^{(\rom{2})}$ satisfies the recurrence relation.
\begin{equation}\label{eq:recu-type1-mean-first}
\begin{split}
\tau_{n}^{(\rom{2})}=2\tau_{n-1}^{(\rom{2})}-\tau_{n-2}^{(\rom{2})},\  \text{for}\ n>k+1,
\end{split}
\end{equation}
with initial conditions $\tau_{n}^{(\rom{2})}=0$, for $0\leq n \leq k-1$, $\tau_{k}^{(\rom{2})}=p^{k}$ and $\tau_{k+1}^{(\rom{2})}=p^{k}(2-p)$.
\end{corollary}

\begin{theorem}
The generating function of the means $E[(T_{r,k}^{(\rom{2})})^{2}]$ is given by
\begin{equation*}\label{eq: 1.1}
\begin{split}
\sum_{n=0}^{\infty}E[(N_{n}^{(\rom{2})})^{2}]z^{n}=\frac{p^{k}z^{k}(1-pz)(1-z+2p^{k}qz^{k+1})}{(1-z)^{3}}.
\end{split}
\end{equation*}
\end{theorem}


\begin{corollary}The $\nu_{n}^{(\rom{2})}=E[(N_{n}^{(\rom{1})})^{2}]$, $n=1,\ 2, ...$ of the random variable $(N_{n}^{(\rom{2})})^{2}$ satisfies the recurrence relation.
\begin{equation}\label{2.1}
\begin{split}
\nu_{n}^{(\rom{2})}=&3\nu_{n-1}^{(\rom{3})}-3\nu_{n-2}^{(\rom{2})}+\nu_{n-3}^{(\rom{2})}+\beta_{n},
\end{split}
\end{equation}
where we set$\nu_{n}^{(\rom{2})}$ for $n<0$ and with initial conditions
\begin{equation*}
\begin{split}
\beta_{n}=\left\{
  \begin{array}{ll}
    0  & \text{if $0\leq n<k$,} \\
    p^{k}  & \text{if $n=k$,} \\
    -p^{k}(1+p)  & \text{if $n=k+1$,} \\
    p^{k+1} & \text{if $n=k+2$,} \\
    2p^{2k}q  & \text{if  $n=2k+1$,} \\
    -2p^{2k+1}q& \text{if $n=2k+2$,} \\
    0& \text{otherwise.} \\
  \end{array}
\right.
\end{split}
\end{equation*}

\end{corollary}

\subsubsection{Overlapping scheme}

\begin{proposition}
The double probability generating function $G(z,w)$ of $N_{n}^{(\rom{3})}$ is given by
\begin{equation}\label{eq:3.1}
\begin{split}
G^{(\rom{3})}(z,w)=\frac{1-pwz-p^{k}(1-w)z^{k}}{1-(1+pw)z+pwz^{2}+p^{k}q(1-w)z^{k+1}}.
\end{split}
\end{equation}
\end{proposition}


In the following lemma we derive a recursive scheme for the evaluation of $G_{n}(w)$.
\begin{lemma}\label{lem:1}
The probability generating function $G_{n}(w)$ of the random variable $N_{n}^{(\rom{3})}$ is satisfies the recursive scheme
\begin{equation}\label{eq: 1.1}
\begin{split}
G_{n}^{(\rom{3})}(w)=&(1+pw)G_{n-1}^{(\rom{3})}(w)-pwG_{n-2}^{(\rom{3})}(w)-p^{k}q(1-w)G_{n-k-1}^{(\rom{3})}(w),\ \text{for}\ n>k,
\end{split}
\end{equation}
with initial conditions
\begin{equation*}
\begin{split}
G_{n}^{(\rom{3})}(w)=\left\{
  \begin{array}{ll}
    1  & \text{if $0\leq n<k$,} \\
    1-p^{k}(1-w)  & \text{if $n=k$.} \\
  \end{array}
\right.
\end{split}
\end{equation*}
\end{lemma}

\begin{proof}
It follows by equating the coefficients of $z^{n}$ on both sides of (\ref{eq:3.1}).
\end{proof}

%
%
%
%
%

An efficient recursive scheme for the evaluation of the probability mass function of $N_{n}^{(\rom{3})}$, ensuing from the result established in \ref{lem:1}, is given in the following theorem.

\begin{theorem}
The probability mass function $g_{n}(x)$ of the random variable $N_{n}^{(\rom{3})}$ satisfies the recursive scheme
\begin{equation}\label{eq: 1.1}
\begin{split}
g_{n}^{(\rom{3})}(x)=&g_{n-1}^{(\rom{3})}(x)+pg_{n-1}^{(\rom{3})}(x-1)-pg_{n-2}^{(\rom{3})}(x-1)-p^{k}qg_{n-k-1}^{(\rom{3})}(x)\\
&+p^{k}qg_{n-k-1}^{(\rom{3})}(x-1),\ \text{for}\ n>k,
\end{split}
\end{equation}
with initial conditions
\begin{equation*}
\begin{split}
g_{n}^{(\rom{3})}(x)=&0,\ \text{for}\ x<0,\ x>n-k+1,\ g_{n}^{(\rom{3})}(x)=\delta_{0,x},\ \text{for}\ 0\leq n<k,\\
g_{k}^{(\rom{3})}(0)=&1-p^{k},\ g_{n}^{(\rom{3})}(1)=p^{k},\ g_{n}^{(\rom{3})}(x)=0,\ \text{for}\ x>1.
\end{split}
\end{equation*}
\end{theorem}

\begin{proof}
It suffices to replace $G_{n}(w)$, in the recursive formulae given in Lemma \ref{lem:1} by the power series
\begin{equation}\label{eq: 3.1}
\begin{split}
G_{n}^{(\rom{3})}(w)=\sum_{x=0}^{\infty}P(X_{n}^{(\rom{3})}=x)w^{x}=\sum_{x=0}^{\infty}g_{n}^{(\rom{3})}(x)w^{x},
\end{split}
\end{equation}
and then equating the coefficients of $w^{x}$ on both sides of the resulting identities.
\end{proof}

Let us give some formulae expressing the generating functions of the first two moments of $N_{n}^{(\rom{3})}$ and $(N_{n}^{(\rom{3})})^{2}$ by means of the double generating function $H^{(\rom{3})}(z,w)$ and $G^{(\rom{3})}(z,w)$.

\begin{theorem}
The generating function of the means $E[N_{n}^{(\rom{3})}]$ is given by
\begin{equation}\label{eq: 1.1}
\begin{split}
\sum_{n=0}^{\infty}E[N_{n}^{(\rom{3})}]z^{n}=\frac{p^{k}z^{k}}{(1-z)^{2}}.
\end{split}
\end{equation}
\end{theorem}


\begin{corollary}The $\tau_{n}^{(\rom{3})}=E[N_{n}^{(\rom{3})}]$, $n=1,\ 2, ...$ of the random variable $N_{n}^{(\rom{3})}$ satisfies the recurrence relation.
\begin{equation}\label{3.1}
\begin{split}
\tau_{n}^{(\rom{3})}=2\tau_{n-1}^{(\rom{3})}-\tau_{n-2}^{(\rom{3})},\  \text{for}\ n>k,
\end{split}
\end{equation}
with initial conditions $\tau_{n}^{(\rom{3})}=0$, for $0\leq n \leq k-1$ and $\tau_{k}^{(\rom{3})}=p^{k}$.
\end{corollary}

\begin{theorem}
The generating function of the means $E[(T_{r,k}^{(\rom{3})})^{2}]$ is given by
\begin{equation*}\label{eq: 1.1}
\begin{split}
\sum_{n=0}^{\infty}E[(N_{n}^{(\rom{3})})^{2}]z^{n}=\frac{p^{k}z^{k}-p^{k+1}z^{k+2}+2p^{2k}qz^{2k+1}}{1-(p+3)z+3(1+p)z^{2}-(1+3p)z^{3}+pz^{4}},
\end{split}
\end{equation*}
\end{theorem}


\begin{corollary}The $\nu_{n}^{(\rom{3})}=E[(N_{n}^{(\rom{3})})^{2}]$, $n=1,\ 2, ...$ of the random variable $(N_{n}^{(\rom{3})})^{2}$ satisfies the recurrence relation.
\begin{equation}\label{2.1}
\begin{split}
\nu_{n}^{(\rom{3})}=&(3+p)\nu_{n-1}^{(\rom{3})}-3(1+p)\nu_{n-2}^{(\rom{2})}+(1+3p)\nu_{n-3}^{(\rom{2})}-p\nu_{n-4}^{(\rom{2})}+\beta_{n},
\end{split}
\end{equation}
where we set $\nu_{n}^{(\rom{3})}$ for $n<0$ and with initial conditions
\begin{equation*}
\begin{split}
\beta_{n}=\left\{
  \begin{array}{ll}
    p^{k}  & \text{if $n=k$,} \\
    0  & \text{if $n=k+1$,} \\
    -p^{k+1} & \text{if $n=k+2$,} \\
    2p^{2k}q& \text{if $n=2k+1$,} \\
    0& \text{otherwise.} \\
  \end{array}
\right.
\end{split}
\end{equation*}
\end{corollary}

\subsection{Markov dependent trials}
\subsubsection{Non-overlapping scheme}

\begin{proposition}
The double probability generating function $G(z,w)$ of $N_{n}^{(\rom{1})}$ is given by
\begin{equation}\label{eq:3.1}
\begin{split}
G^{(\rom{1})}(z,w)=\frac{a_{1}z^{k+1}+a_{2}z^{k}+a_{3}z-1}{1+b_{1}z+b_{2}z^{2}+b_{3}z^{k}+b_{4}z^{k+1}+b_{5}z^{k+2}},
\end{split}
\end{equation}
where
\begin{equation*}\label{eq: 1.1}
\begin{split}
a_{1}=&\alpha^{k}\left\{1-p\beta+(p-\alpha)w\right\},\ a_{2}=\alpha^{k-1}\left\{p+(\alpha-p)w\right\},\ a_{3}=-(1-\alpha-\beta),\\
b_{1}=&-(\alpha+\beta),\ b_{2}=-(1-\alpha-\beta),\ b_{3}=-\alpha^{k}w,\\ b_{4}=&\alpha^{k+1}\left\{(1-\alpha)(1-\beta)-w(1-\alpha-\beta-\alpha^{2})\right\},\ b_{5}=w\alpha^{k}(1-\alpha-\beta).\\
\end{split}
\end{equation*}
\end{proposition}


In the following lemma we derive a recursive scheme for the evaluation of $G_{n}(w)$.
\begin{lemma}\label{lem:1}
The probability generating function $G_{n}(w)$ of the random variable $N_{n}^{(\rom{1})}$ is satisfies the recursive scheme
\begin{equation}\label{eq: 1.1}
\begin{split}
G_{n}^{(\rom{1})}(w)=&(\alpha+\beta)G_{n-1}^{(\rom{1})}(w)+(1-\alpha-\beta)G_{n-2}^{(\rom{1})}(w)+w\alpha^{k}G_{n-k}^{(\rom{1})}(w)\\
&-\alpha^{k+1}\left\{(1-\alpha)(1-\beta)-w(1-\alpha-\beta-\alpha^{2})\right\}G_{n-k-1}^{(\rom{1})}(w)\\
&-w\alpha^{k}(1-\alpha-\beta)G_{n-k-2}^{(\rom{1})}(w).\\
\end{split}
\end{equation}
\end{lemma}

\begin{proof}
It follows by equating the coefficients of $z^{n}$ on both sides of (\ref{eq:3.1}).
\end{proof}

%
%
%
%
%

An efficient recursive scheme for the evaluation of the probability mass function of $N_{n}^{(\rom{1})}$, ensuing from the result established in \ref{lem:1}, is given in the following theorem.

\begin{theorem}
The probability mass function $g_{n}^{(\rom{1})}(x)$ of the random variable $N_{n}^{(\rom{1})}$ satisfies the recursive scheme
\begin{equation}\label{eq: 1.1}
\begin{split}
g_{n}^{(\rom{1})}(x)=&(\alpha+\beta)g_{n-1}^{(\rom{1})}(x)+(1-\alpha-\beta)g_{n-2}^{(\rom{1})}(x)+\alpha^{k}g_{n-k}^{(\rom{1})}(x-1)\\
&-\alpha^{k+1}(1-\alpha)(1-\beta)g_{n-k-1}^{(\rom{1})}(x)+\alpha^{k+1}(1-\alpha-\beta-\alpha^{2})g_{n-k-1}^{(\rom{1})}(x-1)\\
&-\alpha^{k}(1-\alpha-\beta)g_{n-k-2}^{(\rom{1})}(x-1).\\
\end{split}
\end{equation}
\end{theorem}

\begin{proof}
It suffices to replace $G_{n}(w)$, in the recursive formulae given in Lemma \ref{lem:1} by the power series
\begin{equation}\label{eq: 3.1}
\begin{split}
G_{n}^{(\rom{1})}(w)=\sum_{x=0}^{\infty}P(X_{n}^{(\rom{1})}=x)w^{x}=\sum_{x=0}^{\infty}g_{n}^{(\rom{1})}(x)w^{x},
\end{split}
\end{equation}
and then equating the coefficients of $w^{x}$ on both sides of the resulting identities.
\end{proof}

Let us give some formulae expressing the generating functions of the first two moments of $N_{n}^{(\rom{1})}$ and $(N_{n}^{(\rom{1})})^{2}$ by means of the double generating function $H(z,w)$ and $G(z,w)$.

\begin{theorem}
The generating function of the means $E[N_{n}^{(\rom{1})}]$ is given by
\begin{equation}\label{eq: 1.1}
\begin{split}
\sum_{n=0}^{\infty}E[N_{n}^{(\rom{1})}]z^{n}=\frac{\alpha^{k-1}z^{k}(1-\alpha z)\{p+(q-\beta)z\}}{(1-z)^{2}(1-z^{k}\alpha^{k})\{1+(1-\alpha-\beta)z\}}.
\end{split}
\end{equation}
\end{theorem}


\begin{corollary}The $\tau_{n}^{(\rom{1})}=E[N_{n}^{(\rom{1})}]$, $n=1,\ 2, ...$ of the random variable $N_{n}^{(\rom{1})}$ satisfies the recurrence relation.
\begin{equation}\label{eq:recu-type1-mean-first}
\begin{split}
\tau_{n}^{(\rom{1})}=&(1+\alpha+\beta)\tau_{n-1}^{(\rom{1})}+(1-2\alpha-2\beta)\tau_{n-2}^{(\rom{1})}-(1-\alpha-\beta)\tau_{n-3}^{(\rom{1})}+\alpha^{k}\tau_{n-k}^{(\rom{1})}\\
&-\alpha^{k}(1+\alpha+\beta)\tau_{n-k-1}^{(\rom{1})}-\alpha^{k}(1-2\alpha-2\beta)\tau_{n-k-2}^{(\rom{1})}\\
&+\alpha^{k}(1-\alpha-\beta)\tau_{n-k-3}^{(\rom{1})},\  \text{for}\ n>k+2,
\end{split}
\end{equation}
with initial conditions
\begin{equation*}
\begin{split}
\tau_{n}^{(\rom{1})}=\left\{
  \begin{array}{ll}
    0  & \text{if}\ 0\leq n \leq k-1, \\
    p\alpha^{k-1} & \text{if $n=k$,} \\
    \alpha^{k-1}(1-q\beta) & \text{if $n=k+1$,} \\
    \alpha^{k-1}(1-q\beta) & \text{if $n=k+2$.} \\
  \end{array}
\right.
\end{split}
\end{equation*}

\end{corollary}

\begin{theorem}
The generating function of the means $E[(T_{r,k}^{(\rom{1})})^{2}]$ is given by
\begin{equation*}\label{eq: 1.1}
\begin{split}
\sum_{n=0}^{\infty}E[(N_{n}^{(\rom{1})})^{2}]z^{n}=\dfrac{\splitdfrac{\splitdfrac{\alpha^{k-1}z^{k}(\alpha z-1)\{(\beta-q)z-p\}\{\alpha^{k}(1-\alpha-\beta)z^{k+2}}{+\alpha^{k-1}(\alpha^{2}+2\alpha+2\beta-2-\alpha\beta)z^{k+1}-\alpha^{k}z^{k}}}{+(1-\alpha-\beta)z^{2}+(\alpha+\beta)z-1}}{(z-1)^{3}(\alpha^{k}z^{k}-1)^{2}\{(\alpha+\beta-1)z-1\}^{2}},\\
\end{split}
\end{equation*}

\end{theorem}


\subsubsection{Exceed a threashold scheme}
\begin{proposition}
The double probability generating function $G(z,w)$ of $N_{n}^{(\rom{2})}$ is given by
\begin{equation}\label{eq:3.1}
\begin{split}tl
G^{(\rom{2})}(z,w)=\frac{a_{1}z^{k+2}+a_{2}z^{k+1}+a_{3}z^{k}+a_{4}z^{2}+a_{5}z+1}{b_{1}z^{k+2}+b_{2}z^{k+1}+b_{3}z^{3}+b_{4}z^{2}+b_{5}z-1},
\end{split}
\end{equation}
where
\begin{equation*}\label{eq: 1.1}
\begin{split}
a_{1}&=\alpha^{k}(1-w)(q-\beta),\ a_{2}=\alpha^{k-1}(1-w)(p-\alpha+p\alpha+\alpha\beta),\\
a_{3}&=-p\alpha^{k-1}(1-w),\ a_{4}=(\alpha+\beta-1),\ a_{5}=-(\alpha+\beta),\\
b_{1}&=\alpha^{k-1}(1-\alpha)(1-\beta)(1-w),\ b_{2}=-\alpha^{k-1}(1-\alpha)(1-\beta)(1-w),\\
b_{3}&=(\alpha+\beta-1),\ b_{4}=(1-2\alpha-2\beta),\ b_{5}=(1+\alpha+\beta).\\
\end{split}
\end{equation*}
\end{proposition}


In the following lemma we derive a recursive scheme for the evaluation of $G_{n}(w)$.
\begin{lemma}\label{lem:1}
The probability generating function $G_{n}(w)$ of the random variable $N_{n}^{(\rom{2})}$ is satisfies the recursive scheme
\begin{equation}\label{eq: 1.1}
\begin{split}
G_{n}(w)=&(\alpha+\beta+1)G_{n-1}(w)+(1-2\alpha-2\beta)G_{n-2}(w)+(\alpha+\beta-1)G_{n-3}(w)\\
&-\alpha^{k-1}(1-\alpha-w+w\alpha-\beta+\alpha\beta+ w\beta-w\alpha\beta)G_{n-k-1}(w)\\
&+\alpha^{k-1}(1-\alpha-w+w\alpha-\beta+\alpha\beta+ w\beta-w\alpha\beta)G_{n-k-2}(w).\\
\end{split}
\end{equation}
\end{lemma}

\begin{proof}
It follows by equating the coefficients of $z^{n}$ on both sides of (\ref{eq:3.1}).
\end{proof}

%
%
%
%
%

An efficient recursive scheme for the evaluation of the probability mass function of $N_{n}^{(\rom{2})}$, ensuing from the result established in \ref{lem:1}, is given in the following theorem.

\begin{theorem}
The probability mass function $g_{n}(x)$ of the random variable $N_{n}^{(\rom{2})}$ satisfies the recursive scheme
\begin{equation}\label{eq: 1.1}
\begin{split}
g_{n}(x)&(\alpha+\beta+1)g_{n-1}(x)+(1-2\alpha-2\beta)g_{n-2}(x)+(\alpha+\beta-1)g_{n-3}(x)\\
&-\alpha^{k-1}(1-\alpha)(1-\beta)g_{n-k-1}(x)+\alpha^{k-1}(1-\alpha)(1-\beta)g_{n-k-1}(x-1)\\
&+\alpha^{k-1}(1-\alpha)(1-\beta)g_{n-k-2}(x)+\alpha^{k-1}(1-\alpha)(1-\beta)g_{n-k-2}(x-1).\\
\end{split}
\end{equation}
\end{theorem}

\begin{proof}
It suffices to replace $G_{n}(w)$, in the recursive formulae given in Lemma \ref{lem:1} by the power series
\begin{equation}\label{eq: 3.1}
\begin{split}
G_{n}(w)=\sum_{x=0}^{\infty}P(X_{n}=x)w^{x}=\sum_{x=0}^{\infty}g_{n}(x)w^{x},
\end{split}
\end{equation}
and then equating the coefficients of $w^{x}$ on both sides of the resulting identities.
\end{proof}

Let us give some formulae expressing the generating functions of the first two moments of $N_{n}^{(\rom{2})}$ and $(N_{n}^{(\rom{2})})^{2}$ by means of the double generating function $H(z,w)$ and $G(z,w)$.

\begin{theorem}
The generating function of the means $E[N_{n}^{(\rom{2})}]$ is given by
\begin{equation}\label{eq: 1.1}
\begin{split}
\sum_{n=0}^{\infty}E[N_{n}^{(\rom{2})}]z^{n}=\frac{\alpha^{k-1}z^{k}(\alpha z-1)\{(\beta-q)z-p\}}{(z-1)^{2}\{(1-\alpha-\beta)z+1\}}.
\end{split}
\end{equation}
\end{theorem}


\begin{theorem}
The generating function of the means $E[(T_{r,k}^{(\rom{2})})^{2}]$ is given by
\begin{equation*}\label{eq: 1.1}
\begin{split}
\sum_{n=0}^{\infty}E[(N_{n}^{(\rom{2})})^{2}]z^{n}=\dfrac{\splitdfrac{\alpha^{k-2}z^{k}(\alpha z-1)\{(\beta-q)z-p\}\{2\alpha^{k}(1-\alpha)(1-\beta)z^{k+1}}{+\alpha(z-1)\{(1-\alpha\beta)z+1\}}}{(z-1)^{3}\{(1-\alpha-\beta)z+1\}^{2}},\\
\end{split}
\end{equation*}

\end{theorem}


\subsubsection{Overlapping scheme}

\begin{proposition}
The double probability generating function $G(z,w)$ of $N_{n}^{(\rom{3})}$ is given by
\begin{equation}\label{eq:3.1}
\begin{split}
G(z,w)=\frac{a_{1}z^{k+2}+a_{2}z^{k+1}+a_{3}z^{4}+a_{4}z^{3}+a_{5}z^{2}+a_{6}z+1}{b_{1}z^{k+2}+b_{2}z^{k+1}+b_{3}z^{4}+b_{4}z^{3}+b_{5}z^{2}+b_{6}z-1},
\end{split}
\end{equation}
where
\begin{equation*}\label{eq: 1.1}
\begin{split}
a_{1}&=-\alpha^{k-1}(1-\alpha+w\alpha-w-\beta+\alpha\beta-w\alpha\beta+w\beta),\\
a_{2}&=\alpha^{k-1}(1-\alpha+w+w\alpha-\beta+\alpha\beta+\beta w-w\alpha\beta),\ a_{3}=-\alpha w(1-\alpha-\beta)\\
a_{4}&=(1-\alpha+w\alpha-2w\alpha^{2}-\beta-2w\alpha\beta),\\
a_{5}&=-(1-2\alpha-w\alpha-w\alpha^{2}-2\beta-w\alpha\beta),\\
a_{6}&=-(1+\alpha+\beta+w\alpha),\\
b_{1}&=-\alpha^{k-1}(1-\alpha+w\alpha-w-\beta+\alpha\beta-w\alpha\beta+2w\alpha^{2}\beta+w\beta),\\
b_{2}&=\alpha^{k-1}(1-\alpha-w+w\alpha+\beta+w\beta-w\alpha\beta),\ b_{3}=-\alpha w(1-\alpha-\beta)\\
b_{4}&=(1-\alpha+w\alpha-2w\alpha^{2}-\beta-2w\alpha\beta),\\
b_{5}&=(1-2\alpha-w\alpha-w\alpha^{2}-2\beta-w\alpha\beta),\\
b_{6}&=-(1+\alpha+w\alpha+\beta),\\
\end{split}
\end{equation*}
\end{proposition}


In the following lemma we derive a recursive scheme for the evaluation of $G_{n}(w)$.
\begin{lemma}\label{lem:1}
The probability generating function $G_{n}(w)$ of the random variable $N_{n}^{(\rom{3})}$ is satisfies the recursive scheme
\begin{equation}\label{eq: 1.1}
\begin{split}
G_{n}(w)=&+\alpha^{k-1}(1-\alpha-w+w\alpha-\beta+\alpha\beta+ w\beta-w\alpha\beta+2w\alpha^{2}\beta)G_{n-k-2}(w),\\
&-\alpha^{k-1}(1-\alpha-w+w\alpha-\beta+ w\beta-w\alpha\beta)G_{n-k-1}(w),\\
&+\alpha w(1-\alpha-\beta)G_{n-4}(w)-(1-\alpha+w\alpha-2w\alpha^{2}-\beta-2w\alpha\beta)G_{n-3}(w)\\
&+(1-2\alpha-w\alpha-w\alpha^{2}-2\beta-w\alpha\beta)G_{n-2}(w)+(1+\alpha+w\alpha+\beta)G_{n-1}(w).\\
\end{split}
\end{equation}
\end{lemma}

\begin{proof}
It follows by equating the coefficients of $z^{n}$ on both sides of (\ref{eq:3.1}).
\end{proof}

%
%
%
%
%

An efficient recursive scheme for the evaluation of the probability mass function of $N_{n}^{(\rom{2})}$, ensuing from the result established in \ref{lem:1}, is given in the following theorem.

\begin{theorem}
The probability mass function $g_{n}(x)$ of the random variable $N_{n}^{(\rom{3})}$ satisfies the recursive scheme
\begin{equation}\label{eq: 1.1}
\begin{split}
g_{n}(x)=&(\alpha+\beta+1)g_{n-1}(x)+\alpha g_{n-1}(x-1)+(1-2\alpha-2\beta)g_{n-2}(x)\\
&-\alpha(1+\alpha+\beta-1)g_{n-2}(x-1)-(1-\alpha-\beta)g_{n-3}(x)\\
&-\alpha(1-2\alpha-2\beta)g_{n-3}(x-1)+\alpha(1-\alpha-\beta)g_{n-4}(x-1)\\
&-\alpha^{k-1}(1-\alpha+\beta)g_{n-k-1}(x)+\alpha^{k-1}(1-\alpha)(1-\beta)g_{n-k-1}(x-1)\\
&+\alpha^{k-1}(1-\alpha)(1-\beta)g_{n-k-2}(x)\\
&+\alpha^{k-1}(\alpha-1-\alpha\beta+2\alpha^{2}\beta+\beta)g_{n-k-2}(x-1).\\
\end{split}
\end{equation}
\end{theorem}

\begin{proof}
It suffices to replace $G_{n}(w)$, in the recursive formulae given in Lemma \ref{lem:1} by the power series
\begin{equation}\label{eq: 3.1}
\begin{split}
G_{n}(w)=\sum_{x=0}^{\infty}P(X_{n}=x)w^{x}=\sum_{x=0}^{\infty}g_{n}(x)w^{x},
\end{split}
\end{equation}
and then equating the coefficients of $w^{x}$ on both sides of the resulting identities.
\end{proof}

Let us give some formulae expressing the generating functions of the first two moments of $N_{n}^{(\rom{3})}$ and $(N_{n}^{(\rom{3})})^{2}$ by means of the double generating function $H(z,w)$ and $G(z,w)$.

\begin{theorem}
The generating function of the means $E[N_{n}^{(\rom{3})}]$ is given by
\begin{equation}\label{eq: 1.1}
\begin{split}
\sum_{n=0}^{\infty}E[N_{n}^{(\rom{3})}]z^{n}=\frac{\alpha^{k-1}z^{k}\{(q+\beta)z+p\}}{a_{1}z^{k+1}+(\alpha+\beta-1)z^{3}+(1-2\alpha-2\beta)z^{2}+(1+\alpha+\beta)z-1}.
\end{split}
\end{equation}
\end{theorem}


\begin{theorem}
The generating function of the means $E[(T_{r,k}^{(\rom{3})})^{2}]$ is given by
\begin{equation*}\label{eq: 1.1}
\begin{split}
\sum_{n=0}^{\infty}E[(N_{n}^{(\rom{3})})^{2}]z^{n}=\dfrac{\splitdfrac{\alpha^{k-2}z^{k}\{(q-\beta)z+p\}\{2\alpha^{k}(1-\alpha)(1-\beta)z^{k+1}}{+\alpha(z-1)(\alpha z+1)\{(\alpha\beta-1)z-1\}}}{(z-1)^{3}(\alpha z-1)\{(1-\alpha-\beta)z+1\}^{2}}.\\
\end{split}
\end{equation*}

\end{theorem}

\bibliography{biblio}

\begin{thebibliography}{}

\bibitem[Aki et~al., 1996]{aki1996sooner}
Aki, S., Balakrishnan, N., and Mohanty, S. (1996).
\newblock Sooner and later waiting time problems for success and failure runs
  in higher order markov dependent trials.
\newblock {\em Annals of the Institute of Statistical Mathematics},
  48(4):773--787.

\bibitem[Balakrishnan and Koutras, 2002]{balakrishnan2002runs}
Balakrishnan, N. and Koutras, M.~V. (2002).
\newblock {\em Runs and scans with applications}.
\newblock John Wiley \& Sons.

\bibitem[Chang et~al., 2012]{chang2012distribution}
Chang, Y.-M., Fu, J.~C., and Lin, H.-Y. (2012).
\newblock Distribution and double generating function of number of patterns in
  a sequence of markov dependent multistate trials.
\newblock {\em Annals of the Institute of Statistical Mathematics},
  64(1):55--68.

\bibitem[Chaves and de~Souza, 2007]{chaves2007waiting}
Chaves, L.~M. and de~Souza, D.~J. (2007).
\newblock Waiting time for a run of n successes in bernoulli sequences.
\newblock {\em Rev. Bras. Biom}, 25(4):101--113.

\bibitem[Dresden and Du, 2014]{dresden2014simplified}
Dresden, G.~P. and Du, Z. (2014).
\newblock A simplified binet formula for k-generalized fibonacci numbers.
\newblock {\em J. Integer Seq.}, 17(4):14--4.

\bibitem[Feller, 1968]{feller1968introduction}
Feller, W. (1968).
\newblock {\em An introduction to probability theory and its applications},
  volume~1.
\newblock Wiley New York.

\bibitem[Ferguson, 1966]{ferguson1966expression}
Ferguson, D.~E. (1966).
\newblock An expression for generalized fibonacci numbers.
\newblock {\em Fibonacci Quarterly}, 4:270--273.

\bibitem[Flores, 1967]{flores1967direct}
Flores, I. (1967).
\newblock Direct calculation of $k$-generalized fibonacci numbers.
\newblock {\em Fibonacci Quart}, 5(3):259--266.

\bibitem[Gabai, 1970]{gabai1970generalized}
Gabai, H. (1970).
\newblock Generalized fibonacci k-sequences.
\newblock {\em Fibonacci Quart}, 8:31--38.

\bibitem[Greenberg, 1970]{greenberg1970first}
Greenberg, I. (1970).
\newblock The first occurrence of n successes in n trials.
\newblock {\em Technometrics}, 12(3):627--634.

\bibitem[Johnson et~al., 2005]{johnson2005univariate}
Johnson, N.~L., Kotz, S., and Kemp, A.~W. (2005).
\newblock {\em Univariate discrete distributions}.
\newblock John Wiley \& Sons.

\bibitem[Kalman, 1982]{kalman1982generalized}
Kalman, D. (1982).
\newblock Generalized fibonacci numbers by matrix methods.
\newblock {\em Fibonacci Quart}, 20(1):73--76.

\bibitem[Klots and Park, 1972]{klots1972inverse}
Klots, J.~H. and Park, C. (1972).
\newblock Inverse bernoulli trials with dependence.
\newblock Technical report, WISCONSIN UNIV MADISON DEPT OF STATISTICS.

\bibitem[Koutras, 1996]{koutras1996waiting}
Koutras, M. (1996).
\newblock On a waiting time distribution in a sequence of bernoulli trials.
\newblock {\em Annals of the Institute of Statistical Mathematics},
  48(4):789--806.

\bibitem[Koutras, 1997]{koutras1997waiting}
Koutras, M.~V. (1997).
\newblock Waiting times and number of appearances of events in a sequence of
  discrete random variables.
\newblock In {\em Advances in combinatorial methods and applications to
  probability and statistics}, pages 363--384. Springer.

\bibitem[Lee et~al., 2001]{lee2001binet}
Lee, G.-Y., Lee, S.-G., Kim, J.-S., and Shin, H.-K. (2001).
\newblock The binet formula and representations of k-generalized fibonacci
  numbers.
\newblock {\em Fibonacci Quarterly}, 39(2):158--164.

\bibitem[Levesque, 1985]{levesque1985m}
Levesque, C. (1985).
\newblock On m-th order linear recurrences.
\newblock {\em Fibonacci Quart}, 23(4):290--293.

\bibitem[Ling, 1988]{ling1988binomial}
Ling, K. (1988).
\newblock On binomial distributions of order $k$.
\newblock {\em Statistics \& Probability Letters}, 6(4):247--250.

\bibitem[Miles, 1960]{miles1960generalized}
Miles, E. (1960).
\newblock Generalized fibonacci numbers and associated matrices.
\newblock {\em The American Mathematical Monthly}, 67(8):745--752.

\bibitem[Mood, 1940]{mood1940distribution}
Mood, A.~M. (1940).
\newblock The distribution theory of runs.
\newblock {\em The Annals of Mathematical Statistics}, 11(4):367--392.

\bibitem[Philippou and Muwafi, 1982]{philippou1982waiting}
Philippou, A. and Muwafi, A. (1982).
\newblock 1982). waiting for the $k$th consecutive success and the fibonacci
  sequence of order $k$.
\newblock {\em Fibonacci Quarterly}, 20(1):28--32.

\bibitem[Saperstein, 1973]{saperstein1973occurrence}
Saperstein, B. (1973).
\newblock On the occurrence of n successes within n bernoulli trials.
\newblock {\em Technometrics}, 15(4):809--818.

\bibitem[Spickerman and Joyner, 1984]{spickerman1984binet}
Spickerman, W. and Joyner, R. (1984).
\newblock Binet’s formula for the recursive sequence of order k.
\newblock {\em Fibonacci Quarterly}, 22(4):327--331.

\bibitem[Uppuluri and Patil, 1983]{uppuluri1982waiting}
Uppuluri, V. and Patil, S. (1983).
\newblock Waiting times and generalized fibonacci sequences.
\newblock {\em Fibonacci Quarterly}, 21(4):242--249.

\end{thebibliography}
\addcontentsline{toc}{section}{References}
\bibliographystyle{apalike}

\end{document}